\documentclass[psamsfonts]{amsart}  
\usepackage{amssymb}

\usepackage{graphicx,epstopdf}

\usepackage{graphicx}
\usepackage{amssymb}                       
\usepackage{amsmath}
\usepackage{color}
\usepackage{times}

\usepackage[unicode,bookmarks,colorlinks]{hyperref}
\hypersetup{
    linkcolor=brickred,
}

\definecolor{mahogany}{cmyk}{0, 0.77, 0.87, 0}
\definecolor{salmon}{cmyk}{0, 0.53, 0.38, 0}
\definecolor{melon}{cmyk}{0, 0.46, 0.50, 0}
\definecolor{yellowgreen}{cmyk}{0.44, 0, 0.74, 0}
\definecolor{brickred}{cmyk}{0, 0.89, 0.94, 0.28}
\definecolor{OliveGreen}{cmyk}{0.64, 0, 0.95, 0.40}
\definecolor{RawSienna}{cmyk}{0, 0.72, 1.0, 0.45}
\definecolor{ZurichRed}{rgb}{1, 0, 0} 

\usepackage{fancyhdr}
\usepackage{amsmath,amstext,amssymb,amsopn,amsthm}
\usepackage{amsmath,amssymb,amsthm}
\usepackage[mathscr]{eucal}

\begin{document}

\newtheorem{lemma}[thm]{Lemma}
\newtheorem{proposition}{Proposition}
\newtheorem{theorem}{Theorem}[section]
\newtheorem{deff}[thm]{Definition}
\newtheorem{case}[thm]{Case}
\newtheorem{prop}[thm]{Proposition}
\newtheorem{example}{Example}

\newtheorem{corollary}{Corollary}

\theoremstyle{definition}
\newtheorem{remark}{Remark}

\numberwithin{equation}{section}
\numberwithin{definition}{section}
\numberwithin{corollary}{section}

\numberwithin{theorem}{section}

\numberwithin{remark}{section}
\numberwithin{example}{section}
\numberwithin{proposition}{section}

\newcommand{\gap}{\lambda_{2,D}^V-\lambda_{1,D}^V}
\newcommand{\gapR}{\lambda_{2,R}-\lambda_{1,R}}
\newcommand{\bD}{\mathrm{I\! D\!}}
\newcommand{\calD}{\mathcal{D}}
\newcommand{\calA}{\mathcal{A}}

\newcommand{\conjugate}[1]{\overline{#1}}
\newcommand{\abs}[1]{\left| #1 \right|}
\newcommand{\cl}[1]{\overline{#1}}
\newcommand{\expr}[1]{\left( #1 \right)}
\newcommand{\set}[1]{\left\{ #1 \right\}}

\newcommand{\calC}{\mathcal{C}}
\newcommand{\calE}{\mathcal{E}}
\newcommand{\calF}{\mathcal{F}}
\newcommand{\Rd}{\mathbb{R}^d}
\newcommand{\BR}{\mathcal{B}(\Rd)}
\newcommand{\R}{\mathbb{R}}
\newcommand{\T}{\mathbb{T}}
\newcommand{\D}{\mathbb{D}}

\newcommand{\al}{\alpha}
\newcommand{\RR}[1]{\mathbb{#1}}
\newcommand{\bR}{\mathrm{I\! R\!}}
\newcommand{\ga}{\gamma}
\newcommand{\om}{\omega}
\newcommand{\A}{\mathbb{A}}
\newcommand{\bH}{\mathbb{H}}

\newcommand{\bb}[1]{\mathbb{#1}}
\newcommand{\bI}{\bb{I}}
\newcommand{\bN}{\bb{N}}

\newcommand{\uS}{\mathbb{S}}
\newcommand{\M}{{\mathcal{M}}}
\newcommand{\calB}{{\mathcal{B}}}

\newcommand{\W}{{\mathcal{W}}}

\newcommand{\m}{{\mathcal{m}}}

\newcommand {\mac}[1] { \mathbb{#1} }

\newcommand{\bC}{\Bbb C}

\newtheorem{rem}[theorem]{Remark}
\newtheorem{dfn}[theorem]{Definition}
\theoremstyle{definition}
\newtheorem{ex}[theorem]{Example}
\numberwithin{equation}{section}

\newcommand{\Pro}{\mathbb{P}}
\newcommand\F{\mathcal{F}}
\newcommand\E{\mathbb{E}}
\newcommand\e{\varepsilon}
\def\H{\mathcal{H}}
\def\t{\tau}

\title[Fourier multipliers and martingales]{Stability in Burkholder's differentially subordinate martingales  inequalities 
and applications to Fourier multipliers}

\author{Rodrigo Ba\~nuelos}\thanks{R. Ba\~nuelos is supported in part  by NSF Grant 
\# 0603701-DMS}
\address{Department of Mathematics, Purdue University, West Lafayette, IN 47907, USA}
\email{banuelos@math.purdue.edu}
\author{Adam Os\c ekowski}\thanks{A. Os\c ekowski is supported in part by the NCN grant DEC-2014/14/E/ST1/00532.}
\address{Department of Mathematics, Informatics and Mechanics, University of Warsaw, Banacha 2, 02-097 Warsaw, Poland}
\email{ados@mimuw.edu.pl}

\subjclass[2010]{Primary: 60G44, 60G46, Secondary: 42B15, 42B20}
\keywords{martingale, differential subordination, Fourier multipliers}

\begin{abstract}
We study stability estimates for the almost extremal functions associated with the $L^p$-bound for the real and imaginary parts of the Beurling-Ahlfors operator. The proof exploits probabilistic methods and rests on analogous results  for differentially subordinate martingales which are of independent interest.   This allows us to obtain stability inequalities for a larger class of Fourier multipliers.
\end{abstract}

\maketitle

\tableofcontents

\section{Introduction and statements of stability for Fourier multipliers}
Sharp inequalities in analysis and geometry have been of interest for many years and many have been investigated from different points of view where not only their sharpness is proved but the extremal quantities  (those that make the inequality an equality) are identified. Once the extremals are known it is natural to ask about the stability of such inequalities. More specifically, the aim in the investigation of stability inequalities is to measure, in terms of an appropriate  distance  from the extremals, how far an admissible quantity is from attaining equality.   For various examples of such stability results in geometry and spectral theory, we refer the reader to the work by Brasco and  Philippis \cite{BP}.  For a sample of stability inequalities in analysis, see  Bianchi and Egnell \cite{BE},  Chen,  Frank and  Weth \cite{CFW}, Christ \cite{MC1}, Dolbeault and Toscani \cite{DT},  Fathi, Indrei and Ledoux \cite{FIL}, and the very recent paper of Carlen \cite{EC},  to list just a few. 

On the probability side, there has been considerable interest in obtaining sharp inequalities for martingales (for many examples and further references, see the monograph \cite{Os}).  Many of these results have had important applications in analysis; we shall see  some examples  below and relate them to some results from the literature.  It is interesting to note that in the case of many of the  classical martingale inequalities, unlike the inequalities in analysis referenced above,  equality is never attained.  That is, extremals do not exist.  This is the case, for example,  in Doob's maximal inequality, in the classical Burkholder-Davis-Gundy inequalities for martingales with continuous paths (Davis \cite{Da1}),   and in Burkholder's martingale inequalities under the assumption of differential subordination, which include his celebrated sharp martingale transforms inequalities \cite{B}.   In this paper we investigate the stability of Burkholder's inequalities  and apply this to obtain similar results for a class of Fourier multipliers that includes the real and imaginary parts of the Beruling-Ahlfors operator,  the two dimensional Hilbert transform, which has been extensively investigated in the literature.  We also obtain the corresponding result for first order Riesz transforms on $\R^d$, $d\geq 1$.  The latter is new even for the Hilbert transform, the case when $d=1$. 

Our results are motivated from the recent paper \cite{Me} by Melas concerning the structure of almost extremal functions associated with the $L^p$-estimate for the dyadic maximal operator $\mathcal{M}$ on $[0,1]^d$, a version of Doob's maximal inequality. It is well-known that $\mathcal{M}$ is a bounded operator on $L^p([0,1]^d)$, $1<p<\infty$, and its norm equals $p/(p-1)$. Although this norm is never attained, there is a very interesting property of the functions which are almost extremal. A careful inspection of the paper \cite{Me2} reveals that for any $\e>0$ there is $f\in L^p$ for which the pointwise identity  $\mathcal{M}f=\left(\frac{p}{p-1}-\e\right)f$ holds true and therefore this family of functions, corresponding to different $\e$, can be regarded as an ``approximate eigenfunction'' of $\mathcal{M}$ associated with the eigenvalue $p/(p-1)$. One of the main results of \cite{Me} makes this observation more precise. It is proved that if $2<p<\infty$ is a fixed exponent, $\e>0$ is a small number and $f$ is \emph{any} nonnegative function satisfying
$$||\mathcal{M}f||_{L^p([0,1]^d)}\geq\left(\frac{p}{p-1}-\e\right)||f||_{L^p([0,1]^d)},$$
 then 
$$\left|\left|\mathcal{M}f-\frac{p}{p-1}f\right|\right|_{L^p([0,1]^d)}\leq c_p\e^{1/p}||f||_{L^p([0,1]^d)}$$ for some constant $c_p$ depending only on $p$. In other words, if $f$ is almost extremal for the $L^p$-estimate, then it is close, in the $L^p$-sense, to being an eigenfunction of $\mathcal{M}$ corresponding to the eigenvalue $p/(p-1)$. 

A careful analysis  reveals a similar phenomenon for Beurling-Ahlfors operator $B$ on the plane $\mathbb{C}$. Recall that this operator is a Fourier multiplier with the symbol $m(\xi)=\overline{\xi}/\xi$, $\xi\in \mathbb{C}$.  Alternatively, it can be defined by the singular integral operator 
$$ Bf(z)=-\frac{1}{\pi}\,\mbox{p.v.}\int_\mathbb{C} \frac{f(w)}{(z-w)^2}\mbox{d}w.$$
This operator plays a fundamental role in the theory of quasiconformal mappings in the plane.  A convenient reference on the subject is the monograph \cite{AstIwaMar} by Astala, Iwaniec and Martin. A crucial property of $B$ is that it changes the complex derivative $\overline{\partial}$ to $\partial$. More precisely, we have $ B(\overline{\partial} f)={\partial}f$ 
for any $f$ in  the Sobolev space $ W^{1,2}(\mathbb{C},\mathbb{C})$ of complex valued locally integrable functions on $\mathbb{C}$ whose distributional first derivatives are in $L^2$ on the plane.  A beautiful long-standing open problem formulated in 1982  by T. Iwaniec \cite{Iwa} asserts that 
$$ ||B||_{L^p(\mathbb{C})\to L^p(\mathbb{C})}=p^*-1,\qquad 1<p<\infty,$$
where $p^*=\max\{p,p/(p-1)\}$. It is well-known that the $L^p$-norm of $B$ cannot be smaller than $p^*-1$. Curiously, the almost-extremal functions, constructed by Lehto \cite{Le}, are also close to being eigenfunctions of $B$, but \emph{up to absolute value}. To state this more  precisely, suppose first that $1<p\leq 2$. For a given $\beta\in (-2/p,0)$, let $f_\beta(z)=|z|^\beta\chi_{\mathbb{D}}(z)$, where $\mathbb{D}$ is the unit disc in the plane. Using the commutation of $\overline{\partial}$ and $\partial$ by $B$, we have that 
\begin{align*}
 Bf_\beta(z)&=B\left(\bar{\partial}\left(\frac{2|z|^\beta\bar{z}}{\beta+2}\chi_{\mathbb{D}}+\frac{2z^{-1}}{\beta+2}\chi_{\mathbb{C}\setminus \mathbb{D}}\right)\right)\\
 &=\partial \left(\frac{2|z|^\beta\bar{z}}{\beta+2}\chi_{\mathbb{D}}+\frac{2z^{-1}}{\beta+2}\chi_{\mathbb{C}\setminus \mathbb{D}}\right)\\
&=\frac{\beta |z|^\beta \bar{z}/z}{\beta+2}\chi_{\mathbb{D}}-\frac{2z^{-2}}{\beta+2}\chi_{\mathbb{C}\setminus \mathbb{D}}=\frac{\beta \bar{z}/z}{\beta+2}f_\beta-\frac{2z^{-2}}{\beta+2}\chi_{\mathbb{C}\setminus \mathbb{D}}.
\end{align*}
Now, if we let $\beta\downarrow -2/p$, then $|\beta/(\beta+2)|\to (p-1)^{-1}=p^*-1$.  Furthermore, the $L^p$-norm of $f_\beta$ converges to infinity and the ``error term'' $-\frac{2z^{-2}}{\beta+2}\chi_{\mathbb{C}\setminus \mathbb{D}}$ becomes irrelevant, so that $||Bf_\beta||_{L^p(\mathbb{C})}/||f_\beta||_{L^p(\mathbb{C})}\to p^*-1$. However, the above formula shows that, essentially, $|Bf_\beta|\approx (p^*-1)|f_\beta|$ pointwise, provided $\beta$ is close to $-2/p$.  In other words, $f_\beta$ is almost an eigenfunction of $B$ with the eigenvalue $p^*-1$, up to absolute value. In the case $p>2$ the calculations are similar and exploit the functions 
$$ f_\beta(z)=\frac{\beta |z|^\beta z/\bar{z}}{\beta+2}\chi_{\mathbb{D}}-\frac{2\bar{z}^{-2}}{\beta+2}\chi_{\mathbb{C}\setminus \mathbb{D}},\qquad \beta\in (-2/p,0),$$
for which $Bf_\beta(z)=|z|^\beta\chi_{\mathbb{D}}$ and $|Bf_\beta|\approx (p^*-1)|f_\beta|$ provided $\beta$ is close to $-2/p$.

Our contribution in this paper is to present a quantitative version of stability result for a large class of Fourier multipliers which includes the real and imaginary parts of Beurling-Ahlfors operator and first order Riesz transforms. Consider the following class of symbols, introduced by 
in \cite{BBB}. 
Assume that $\mu$ is a finite nonnegative Borel measure on the unit sphere $\mathbb{S}$ of $\R^d$ and fix a Borel function 
$\psi$ on $\mathbb{S}$ which take values in the unit ball of $\mathbb{C}$. We define the associated multiplier $m=m_{\psi,\mu}$ on $\R^d$ by
\begin{equation}\label{defm}
 m(\xi)=\frac{\int_\mathbb{S} \langle \xi,\theta\rangle^2\psi(\theta)\mu(\mbox{d}\theta)}
{\int_\mathbb{S} \langle \xi,\theta\rangle^2\mu(\mbox{d}\theta)}
 \end{equation}
if the denominator is not $0$, and $m(\xi)=0$ otherwise. Here $\langle\cdot,\cdot\rangle$ stands for the scalar product on $\R^d$. This class is quite large, containing the real and imaginary parts of the Beurling-Ahlfors operator. To see this, note that $B$  can be decomposed as $B=R_2^2-R_1^2-2iR_1R_2$, where $R_1$, $R_2$ are planar Riesz transforms, that is, the Fourier multipliers with the symbols $-i\xi_1/|\xi|$ and $-i\xi_2/|\xi|$, respectively (see the discussion following Theorem \ref{mainthmf} below).  Indeed, we have the identity
$$ \frac{\overline{\xi}}{\xi}=\frac{\xi_1^2-\xi_2^2}{\xi_1^2+\xi_2^2}-i\frac{2\xi_1\xi_2}{\xi_1^2+\xi_2^2}.$$
Now  both $R_2^2-R_1^2$ and $2R_1R_2$ can be represented as the Fourier multipliers with the symbols of the form \eqref{defm}: the choice $d=2$, $\mu=\delta_{(1,0)}+\delta_{(0,1)}$, $\psi(1,0)=-1=-\psi(0,1)$ leads to $T_m=\Re B$, while taking $d=2$, $\mu=\delta_{(1/\sqrt{2},1/\sqrt{2})}+\delta_{(1/\sqrt{2},-1/\sqrt{2})}$ and $\psi(1/\sqrt{2},1/\sqrt{2})=1=\psi(1/\sqrt{2},-1/\sqrt{2})$ yields $T_m=-\Im B$.  

One of the main results of \cite{BBB} is that if $m$ is as above, then
\begin{equation}\label{BanBog}
 ||T_m||_{L^p(\R^d)\to L^p(\R^d)}\leq p^*-1.
\end{equation}
Furthermore, as shown by Geiss, Montgomery-Smith and Saksman \cite{GMS}, equality holds for the real and imaginary parts of $B$.    We also refer to  \cite{BO} for other such examples. (The bound in \eqref{BanBog} for the real and imaginary parts of $B$ was  proved by  Nazarov and  Volberg \cite{NazVol}, see also \cite{BanMen}.)  One of our main results concerns the $L^p$-stability of such multipliers. Here is the precise statement.

\begin{theorem}\label{mainthmf}
Suppose that $m$ is a symbol from the class \eqref{defm} and $T_m$ is the associated Fourier multiplier.

(i) Let $1<p<2$ and $\e>0$. If $f$ is such that $$||T_mf||_{L^p(\R^d)}\geq ((p-1)^{-1}-\e)||f||_{L^p(\R^d)},$$ then
\begin{equation}\label{mainf<2}
\big|\big| |T_mf|-(p-1)^{-1}|f|\big|\big|_{L^p(\R^d)}\leq c_p\e^{1/2}||f||_{L^p(\R^d)},
\end{equation}
where
$$ c_p=\frac{\left(\frac{p}{p-1}\right)^{(3-p)/2}}{(1-p\left(1-1/p\right)^{p-1})^{1/2}}.$$
The order $O(\e^{1/2})$, as $\e\to 0$, is optimal. Furthermore, the multiplicative factor $c_p$ is of optimal order $O((2-p)^{-1/2})$, as $p\uparrow 2$.

(ii) Let $2<p<\infty$ and $ \e>0$. If $f$ is such that $$||T_mf||_{L^p(\R^d)}\geq (p-1-\e)||f||_{L^p(\R^d)},$$ then
\begin{equation}\label{mainf>2}
\big|\big| |T_mf|-(p-1)|f|\big|\big|_{L^p(\R^d)}\leq c_p\e^{1/p}||f||_{L^p(\R^d)},
\end{equation}
where 
$$ c_p=(p-1)\left[\frac{2pe}{(p-2)(e-2)}\right]^{1/p}.$$
The order $O(\e^{1/p})$, as $\e\to 0$,  is optimal. Furthermore, the multiplicative constant $c_p$ is of optimal orders $O((p-2)^{-1/p})$, as $p\downarrow 2$, and $O(p)$, as $p\to \infty$.

(iii) For $p=2$, there is no stability result of the above type.  That is,  there are no finite constants $c_2$ and $\kappa>0$ such that $$\big|\big| |T_mf|-|f|\big|\big|_{L^2(\R^d)}\leq c_2\e^{\kappa}||f||_{L^2(\R^d)}$$ provided $$||T_mf||_{L^2(\R^d)}\geq (1-\e)||f||_{{L^2(\R^d)}},$$ with $\e$ sufficiently small.
\end{theorem}

A few remarks are in order. First, it is clear that the above statement is meaningful only for multipliers which have $L^p$ norm equal to $p^*-1$. Furthermore, the aforementioned  optimality of the constants and exponents will be shown for the real part of the Beurling-Ahlfors operator (and a similar reasoning proves that the optimality holds also for the imaginary part of $B$). We do not know whether the order of $c_p$ as $p\downarrow 1$ is optimal; our examples below indicate that $c_p\geq O((p-1)^{-1})$.

We also mention that inequalities  \eqref{mainf<2} and \eqref{mainf>2} hold for the class of Calder\'on--Zygmund singular integrals  $T_Af(x)=\int_{\R^d} K_A(x, y)f(y)dy$ with kernels of the form 
\begin{equation}
K_A(x, y) = \int_0^\infty \int_{\mathbb{R}^d} \left(A(\bar{x},t)\nabla_{x} p_t(\bar{x}-y)\right) \cdot \nabla_{x} p_t(\bar{x}-x) d\bar{x}dt.  
\end{equation}
Here $A(x, t)$ is an $d\times d$ matrix-valued function with 
$$ \|A\|=\|\sup_{|v|\leq 1}(|A(x,y)v|)\|_{L^\infty(\mathbb{R}^d\times [0,\infty))} \leq 1$$ 
and $\nabla_{x} p_t$ denotes the gradient of the Gaussian (heat) kernel $p_t$.  These are Calder\'on-Zygmund operators but not of convolution (or Fourier multiplier) type unless the matrix does not depend on $x$.  They arise from martingale transforms and as in the case of the multipliers in Theorem \ref{mainthmf},  their $L^p$-norms are also bounded above by $(p^*-1)$.  For details, we refer to Perlmutter \cite{Per}.

Next we present a version of the above result for first-order Riesz transforms. Recall that for any dimension $d\geq 1$, the family of Riesz transforms on $\R^d$ is given by
$$ R_jf(x)=\frac{\Gamma\left(\frac{d+1}{2}\right)}{\pi^{(d+1)/2}}\,\int_{\R^d} \frac{x_j-y_j}{|x-y|^{d+1}}f(y)\mbox{d}y,\qquad j=1,\,2,\,\ldots,\,d,$$
where the integrals are supposed to exist in the sense of Cauchy principal values. In the particular case $d=1$, the family consists of only one element, the Hilbert transform $\mathcal{H}$ on $\R$. Alternatively, $R_j$ can be defined as the Fourier multiplier with the symbol $-i\xi_j/|\xi|$, $\xi\in \R^d\setminus\{0\}$. 
As proved by Iwaniec and Martin \cite{IM}, for any $1<p<\infty$ and any $f\in L^p(\R^d)$ we have
\begin{equation}\label{inn}
||R_j f||_{L^p(\R^d)}\leq \cot\frac{\pi}{2p^*} ||f||_{L^p(\R^d)},\qquad j=1,\,2,\,\ldots,\,d,
\end{equation}
and the constant cannot be decreased. (Note that $ \cot\frac{\pi}{2p^*}$ equals $\tan\frac{\pi}{2p}$ if $1<p\leq 2$, and $\cot\frac{\pi}{2p}$ if $p\geq 2$.) An alternative probabilistic proof of the estimate \eqref{inn} based on a sharp estimate for orthogonal martingales, was given 
in \cite{BW}. Our contribution in this direction is the following stability result.

\begin{theorem}\label{mainthmR}
Let $d$ be a fixed positive integer and let $j\in \{1,\,2,\,\ldots,\,d\}$. Furthermore, let $1<p<\infty$ and pick $f\in L^p(\R^d)$.

(i) Suppose that $1<p<2$ and let $\e>0$. If $f$ is such that $$||R_jf||_{L^p(\R^d)}\geq (\tan\frac{\pi}{2p}-\e)||f||_{L^p(\R^d)},$$ then
\begin{equation}\label{mainR<2}
\left|\left| |R_jf|-\tan\frac{\pi}{2p}|f|\right|\right|_{L^p(\R^d)}\leq c_p\e^{1/2}||f||_{L^p(\R^d)},
\end{equation}
where
$$ c_p=\left(\frac{32}{\pi}\right)^{1/2}\frac{p^{(3-p)/2}}{(p-1)^{(4-p)/2}(2-p)^{1/2}}.$$
The order $O(\e^{1/2})$, as $\e\to 0$, is optimal. Furthermore, the multiplicative factor $c_p$ is of optimal order $O((2-p)^{-1/2})$, as $p\uparrow 2$.

(ii) Suppose that $2<p<\infty$ and let $ \e>0$. If $f$ is such that $$||R_jf||_{L^p(\R^d)}\geq (\cot\frac{\pi}{2p}-\e)||f||_{L^p(\R^d)},$$ then
\begin{equation}\label{mainR>2}
\left|\left| |R_jf|-\cot\frac{\pi}{2p}|f|\right|\right|_{L^p(\R^d)}\leq c_p\e^{1/p}||f||_{L^p(\R^d)},
\end{equation}
where 
$$ c_p=(p-1)\left[\frac{(2+\sqrt{2})p^2}{(p-1)(p-2)}\right]^{1/p}.$$
The order $O(\e^{1/p})$, as $\e\to 0$, is optimal. Furthermore, the multiplicative constant $c_p$ is of optimal orders $O((p-2)^{-1/p})$, as $p\downarrow 2$ and $O(p)$ as $p\to \infty$.

(iii) For $p=2$, there are no finite positive constants $c_2$ and $\kappa$ such that for sufficiently small $\e>0$, the inequality $||R_jf||_{L^2(\R^d)}\geq (1-\e)||f||_{L^2(\R^d)}$ implies $\big|\big| |R_jf|-|f|\big|\big|_{L^2(\R^d)}\leq c_2\e^{\kappa}||f||_{L^2(\R^d)}$.
\end{theorem}

 As noted above, when $d=1$, $R_1$ reduces to the classical Hilbert transform and Theorem \ref{mainthmR} gives the stability of Pichorides'  \cite{P} inequality.

Let us say a few words about the proofs and the organization of the paper. Our approach will be probabilistic and will exploit similar  tight estimates for differentially subordinate martingales. The probabilistic content of the paper can be found in \S2, while \S3 contains the proofs of the analytic results, Theorem \ref{mainthmf} and Theorem \ref{mainthmR}.

\section{Stability for martingale inequalities}

\subsection{Background, statement of results and  method of proofs}

Suppose that $(\Omega,\calF,\mathbb{P})$ is a complete probability space, filtered by $(\calF_t)_{t\geq 0}$, a non-decreasing family of sub-$\sigma$-algebras of $\calF$ such that $\calF_0$ contains all the events of probability $0$. Let $X$, $Y$ be two adapted c\'adl\'ag martingales, i.e., with right-continuous trajectories that have limits from the left, taking values in  a given separable Hilbert space $\mathbb{H}$. We may and will assume that $\mathbb{H}$ is equal to $\ell_2$, and we will denote the norm in $\mathbb{H}$ by $|\cdot|$, and the corresponding inner product by $\langle \cdot,\cdot\rangle$. The symbol $[X,X]$ stands for the square bracket of $X$; see e.g. Dellacherie and Meyer \cite{DM} for the definition in the real-valued case, and extend the notion to the vector setting by $[X,X]=\sum_{k=1}^\infty [X^k,X^k]$, where $X^k$ is the $k$-th coordinate of $X$. For any $1\leq p\leq \infty$, we will use the notation $||X||_p=\sup_{t\geq 0}||X_t||_p$ for the $p$-th norm of the process $X$, and denote by $X_\infty$ the almost sure limit $\lim_{t\to\infty}X_t$, if it exists. Martingales $X$ and $Y$ are said to be orthogonal, if their square bracket is constant: $[X,Y]=[X,Y]_0$. Following 
\cite{BW} and 
 \cite{W}, we say that $Y$ is \emph{differentially subordinate} to $X$, if the process $([X,X]_t-[Y,Y]_t)_{t\geq 0}$ is nonnegative and nondecreasing as a function of $t$. The origins of this notion go back to Burkholder's paper \cite{B-1}, who introduced the differential subordination in the context of discrete martingales: a martingale $g=(g_n)_{n\geq 0}$ is differentially subordinate to $f=(f_n)_{n\geq 0}$ if we have $|g_0|\leq |f_0|$ and $|g_n-g_{n-1}|\leq |f_n-f_{n-1}|$ almost surely for all $n$. Treating such martingales as continuous-time processes (via $X_t=f_{\lfloor t\rfloor}$, $Y_t=g_{\lfloor t\rfloor}$), we see that the continuous-time definition  is consistent with the original one. The following discrete-time example will be of importance to us later: suppose that $f=(f_n)_{n\geq 0}$ is a martingale and let $v=(v_n)_{n\geq 0}$ be a deterministic sequence. We say that $g$ is the transform of $f$ by $v$ if we have $g_0=v_0f_0$ and $g_n-g_{n-1}=v_n(f_n-f_{n-1})$ for all $n\geq 1$. One immediately checks that if the sequence $v$ takes values in the interval $[-1,1]$, then $g$ is differentially subordinate to $f$.

Differential subordination (regardless of orthogonality) implies many interesting inequalities between the processes involved, and these estimates have plenty of further applications in many areas of mathematics. The literature on this is now quite large, we refer the reader to the works \cite{BanBau}, \cite{BB}, \cite{BW}, \cite{B}, \cite{B1}, \cite{B2}, \cite{BO1}, \cite{BO2}, \cite{Os}, \cite{Os3}, \cite{W} and references therein. For example, we have the following classical statement, proved by Burkholder \cite{B} in the discrete-time setting and extended to the continuous time by Wang \cite{W}. We keep the notation $p^*=\max\{p,p/(p-1)\}$ introduced in the preceding section.

\begin{theorem}
Suppose that $X$, $Y$ are $\mathbb{H}$-valued martingales such that $Y$ is differentially subordinate to $X$. Then for any $1<p<\infty$ we have the inequality
\begin{equation}\label{burkin}
||Y||_p\leq (p^*-1)||X||_p.
\end{equation}
The constant $p^*-1$ is the best possible even in the above context of discrete-time martingale transforms with $\mathbb{H}=\R$.
\end{theorem}

Here is the orthogonal version of the above statement, proved by 
 in \cite{BW}. In the case when the martingales arise from conjugate harmonic functions in the disc, this is due to   Pichorides \cite{P}.

\begin{theorem}
Suppose that $X$, $Y$ are real-valued orthogonal martingales such that $Y$ is differentially subordinate to $X$. Then for any $1<p<\infty$ we have the inequality
\begin{equation}\label{BWin}
||Y||_p\leq \cot\frac{\pi}{2p^*}||X||_p.
\end{equation}
The constant $\cot\frac{\pi}{2p^*}$ is the best possible.
\end{theorem}

For the vector-valued version of this result, consult the work \cite{O4}.  In this setting  the constants change slightly in the case $1<p<3$.

One of our main results is the following $L^p$-stability statement in the above probabilistic context.  This  can be regarded as the stochastic analogue of Theorem \ref{mainthmf}.

\begin{theorem}\label{mainthm}
Suppose that $X$, $Y$ are $\mathbb{H}$-valued martingales such that $Y$ is differentially subordinate to $X$. 

(i) Let $1<p<2$ and $\e>0$. If $X$ and $Y$ are $L^p$-bounded and satisfy the estimate $||Y||_p\geq ((p-1)^{-1}-\e)||X||_p$, then
\begin{equation}\label{main<2}
\big|\big| |Y_\infty|-(p-1)^{-1}|X_\infty|\big|\big|_p\leq c_p\e^{1/2}||X||_p,
\end{equation}
where
$$ c_p=\frac{\left(\frac{p}{p-1}\right)^{(3-p)/2}}{(1-p\left(1-1/p\right)^{p-1})^{1/2}}.$$
The order $O(\e^{1/2})$ as $\e\to 0$ is optimal. Furthermore, the multiplicative factor $c_p$ is of optimal order $O((2-p)^{-1/2})$ as $p\uparrow 2$.

(ii) Let $2<p<\infty$ and $ \e>0$. If $X$ and $Y$ are $L^p$-bounded and satisfy the estimate $||Y||_p\geq (p-1-\e)||X||_p$, then
\begin{equation}\label{main>2}
\big|\big| |Y_\infty|-(p-1)|X_\infty|\big|\big|_p\leq c_p\e^{1/p}||X||_p,
\end{equation}
where 
$$ c_p=(p-1)\left[\frac{2pe}{(p-2)(e-2)}\right]^{1/p}.$$
The order $O(\e^{1/p})$ as $\e\to 0$ is optimal. Furthermore, the multiplicative constant $c_p$ is of optimal orders $O((p-2)^{-1/p})$ as $p\downarrow 2$ and $O(p)$ as $p\to \infty$.

(iii) For $p=2$, there are no finite positive constants $c_2$ and $\kappa$ such that for sufficiently small $\e>0$, the inequality $||Y||_2\geq (1-\e)||X||_2$ implies $\big|\big| |Y_\infty|-|X_\infty|\big|\big|_2\leq c_2\e^{\kappa}||X||_2$.
\end{theorem}

Note that the assumption on the $L^p$-boundedness of $X$ and $Y$ implies the existence of $X_\infty$ and $Y_\infty$, so the above formulation makes sense. 
As in the analytic setting, we do not know whether the order $O((p-1)^{-3/2})$ of $c_p$ as $p\downarrow 1$ is optimal. We will prove that $c_p\geq O((p-1)^{-1})$, by constructing appropriate examples.

Let us briefly handle the case $p=2$. Suppose that $\Omega=[0,1]$, $\F=\mathcal{B}(0,1)$ and $\mathbb{P}$ is a Lebesgue measure. Take $X_t=Y_t=\chi_{[0,1]}$ for $t\in [0,1)$ and $X_t=2\chi_{[0,1/2]}$, $Y_t=2\chi_{(1/2,1]}$, $t\geq 1$. Then $Y$ is differentially subordinate to $X$ (which is equivalent to the trivial inequality $|Y_1-Y_{1-}|\leq |X_1-X_{1-}|$). Furthermore, we have $||Y||_2=||X||_2$, so the condition $||Y||_2\geq (1-\e)||X||_2$ is satisfied for all $\e>0$; on the other hand, the ratio $\big|\big||Y_\infty|-|X_\infty|\big|\big|_2/||X||_2$ is positive (and does not depend on $\e$), so the inequality $\big|\big| |Y_\infty|-|X_\infty|\big|\big|_2\leq c_2\e^{\kappa}||X||_2$ is violated for sufficiently small $\e$ (no matter what $c_2$ and $\kappa$ are). Therefore, we have to establish the first two parts of Theorem \ref{mainthm}, and this will be done in Subsections \ref{example}, \ref{p<2no} and \ref{p>2no} below.

In the orthogonal case we will prove the following statement.

\begin{theorem}\label{mainthmo}
Suppose that $X$, $Y$ are orthogonal real-valued martingales such that $Y$ is differentially subordinate to $X$. 

(i) Let $1<p<2$ and $\e>0$. If $X$ and $Y$ are $L^p$-bounded and satisfy the estimate $||Y||_p\geq (\tan\frac{\pi}{2p}-\e)||X||_p$, then
\begin{equation}\label{maino<2}
\left|\left| |Y_\infty|-\tan\frac{\pi}{2p}|X_\infty|\right|\right|_p\leq c_p\e^{1/2}||X||_p,
\end{equation}
where
$$ c_p=\left(\frac{32}{\pi}\right)^{1/2}\frac{p^{(3-p)/2}}{(p-1)^{(4-p)/2}(2-p)^{1/2}}.$$
The order $O(\e^{1/2})$ as $\e\to 0$ is optimal. Furthermore, the multiplicative factor $c_p$ is of optimal order $O((2-p)^{-1/2})$ as $p\uparrow 2$.

(ii) Let $2<p<\infty$ and $ \e>0$. If $X$ and $Y$ are $L^p$-bounded and satisfy the estimate $||Y||_p\geq (\cot\frac{\pi}{2p}-\e)||X||_p$, then
\begin{equation}\label{maino>2}
\left|\left| |Y_\infty|-\cot\frac{\pi}{2p}|X_\infty|\right|\right|_p\leq c_p\e^{1/p}||X||_p,
\end{equation}
where 
$$ c_p=(p-1)\left[\frac{(2+\sqrt{2})p^2}{(p-1)(p-2)}\right]^{1/p}.$$
The order $O(\e^{1/p})$ as $\e\to 0$ is optimal. Furthermore, the multiplicative constant $c_p$ is of optimal orders $O((p-2)^{-1/p})$ as $p\downarrow 2$ and $O(p)$ as $p\to \infty$.

(iii) For $p=2$, there are no finite positive constants $c_2$ and $\kappa$ such that for sufficiently small $\e>0$, the inequality $||Y||_2\geq (1-\e)||X||_2$ implies $\big|\big| |Y_\infty|-|X_\infty|\big|\big|_2\leq c_2\e^{\kappa}||X||_2$.
\end{theorem}

As in the non-orthogonal case, the third part of the above theorem is easy. For example, consider a two-dimensional Brownian motion $(X,Y)$ started at the origin and stopped upon reaching the boundary of the unit disc. Then the inequality $||Y||_2\geq (1-\e)||X||_2$ is satisfied for all $\e>0$, while $|||Y_\infty|-|X_\infty|||_2>0$, so the inequality $\big|\big| |Y_\infty|-|X_\infty|\big|\big|_2\leq c_2\e^{\kappa}||X||_2$ does not hold for sufficiently small $\e$, regardless of the values of $c_2$ and $\kappa$. Thus, we need to prove (i) and (ii), which will done below in Subsections \ref{p<2o} and \ref{p>2o}.

Let us now describe our approach. The proof of the inequalities \eqref{main<2} and \eqref{main>2} will be based on Burkholder's method (or Bellman function method): we will deduce the validity of these estimates from the existence of certain special functions, satisfying appropriate majorization and concavity. See \cite{B1} or \cite{Os} for the detailed description of the technique. Our approach exploits  the following statements, which are slight generalizations of the results of Wang \cite{W} (see Lemma 3 and Proposition 1 there).

\begin{theorem}\label{Wang}
Let $U$ be a continuous function on $\mathbb{H}\times \mathbb{H}$ satisfying the following conditions.

(i) The function $U$ is bounded on bounded sets, is of class $C^1$ on $\mathbb{H}\times \mathbb{H}\setminus \{|x||y|=0\}$ and of class $C^2$ on $S_i$, $i\geq 1$, where $S_i$ is a sequence of open connected sets such that the union of closures of $S_i$ is $\mathbb{H}\times \mathbb{H}$.

(ii) For each $i$, there is a nonnegative measurable function $c_i$ on $S_i$ such that for any $(x,y)\in S_i$ and any $h,\,k\in\mathbb{H}$,
\begin{equation}\label{conv}
\langle U_{xx}(x,y)h,h\rangle+2\langle U_{xy}(x,y)h,k\rangle+\langle U_{yy}(x,y)k,k\rangle\leq -c_i(x,y)(|h|^2-|k|^2).
\end{equation} 
Furthermore, for each $i$ and $n$ there is a finite constant $M_{i,n}$ such that 
$$ \sup\{c_i(x,y):(x,y)\in S_i,\,1/n\leq |x|+|y|<n\}\leq M_{i,n}.$$

\noindent Then for $t\geq 0$ and any pair $X$, $Y$ of $\mathbb{H}$-valued martingales such that $Y$ is differentially subordinate to $X$, there is a nondecreasing sequence $(\tau_n)_{n\geq 1}$ of stopping times converging to infinity such that
$$ \E U(X_{\tau_n\wedge t},Y_{\tau_n\wedge t})\leq \E U(X_0,Y_0),\qquad n=1,\,2,\,\ldots.$$
\end{theorem}

\begin{theorem}\label{Wang2}
Let $U$ be a continuous function on $\R^2$ satisfying the following conditions.

(i) The function $U$ is bounded on bounded sets and of class $C^1$ on $\R^2\setminus \{|x||y|=0\}$.

(ii) The function $U$ is superharmonic and, for any fixed $x$, the function $U(x,\cdot)$ is convex. 

\noindent Then for $t\geq 0$ and any pair $X$, $Y$ of real-valued orthogonal martingales such that $Y$ is differentially subordinate to $X$, there is a nondecreasing sequence $(\tau_n)_{n\geq 1}$ of stopping times converging to infinity such that
$$ \E U(X_{\tau_n\wedge t},Y_{\tau_n\wedge t})\leq \E U(X_0,Y_0),\qquad n=1,\,2,\,\ldots.$$
\end{theorem}

For the proof, one needs to repeat the reasoning appearing in \cite{BW} and \cite{W} (see also \cite{Os}); we will omit this argumentation, leaving it to the interested reader.

\subsection{Sharpness of the martingale inequalities,  a discrete-time example}\label{example} Our starting point is the construction of a certain special discrete-time martingale pair, which will be used in both cases $p<2$ and $p>2$ of Theorem \ref{mainthm}. For the sake of clarity and to ease the computations, we have decided to split the construction into two stages. Fix a large positive number $K>1$, a large positive integer $N$, a small number $\eta>0$ and set $\delta=(K^{1/N}-1)/2$, so that $(1+2\delta)^N=K$ and $2N\delta\approx \log K$ for large $N$.

\emph{First stage.}  Consider the Markov martingale $(F,G)$ with a distribution uniquely determined by the following requirements.

\begin{itemize}
\item[(i)] $(F_0,G_0)\equiv (0,0)$, $(F_1,G_1)\in \left\{(1/2,-1/2),(-1/2,1/2)\right\}$.

\item[(ii)] For $y\neq 0$, the point $(y,-y)$ leads to $(2y,0)$ or to $(0,-2y)$.

\item[(iii)] For $|y|< K$, the point $(0,y)$ leads to $(y/p,(p-1)y/p)$ or to $(-\delta y,y+\delta y)$.

\item[(iv)] For $y\neq 0$, the point $(-\delta y,y+\delta y)$ leads to $(0,y+2\delta y)$ or to $(-(y+2\delta y)/p,(p-1)(y+2\delta y)/p)$.

\item[(v)] All the remaining points are absorbing.
\end{itemize}

\begin{figure}[htbp]
\begin{center}
\includegraphics[scale=0.3]{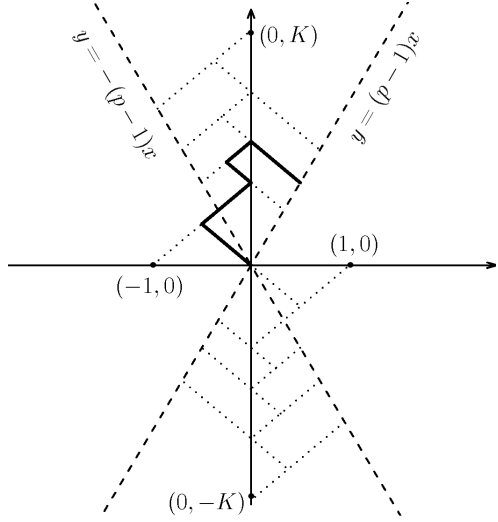}

\caption{
Markov martingale $(F,G)$, first stage. The dotted lines describe the possible directions for the evolution of the process. The bold line refers to an exemplary trajectory of $(F,G)$, which starts from $(0,0)$ and moves to $(-1/2,1/2)$, $(0,1)$, $(-\delta,1+\delta)$, $(0,1+2\delta)$ and $((1+2\delta)/p,(p-1)(1+2\delta)/p)$ in the consecutive steps.}\label{traject}
\end{center}
\end{figure} 

Some remarks are in order. First, we do not need to specify the transition probabilities, they are uniquely determined by the condition that $(F,G)$ is a martingale. Note that $G$ is the transform of $F$ by the deterministic sequence $\{(-1)^n\}_{n=0}^\infty$. Clearly, this condition is symmetric: $F$ is the transform of $G$ by the same deterministic sequence. We will exploit this symmetry later on. Finally, observe that the martingale pair $(F,G)$ is finite in the sense that it terminates after a finite number of steps.

Let us look at the distribution of the random variable $(|F_\infty|,|G_\infty|)$ (note that norms are applied to both $F_\infty$ and $G_\infty$). It takes the value $(1,0)$ with the probability
\begin{equation}\label{prob0}
 \mathbb{P}\big((|F_\infty|,|G_\infty|)=(1,0)\big)=\mathbb{P}\big((|F_2|,|G_2|)=(1,0)\big)=\frac{1}{2}.
\end{equation}
Furthermore, we have
\begin{equation}\label{prob1}
 \mathbb{P}\big((|F_\infty|,|G_\infty|)=(1/p,(p-1)/p)\big)=\mathbb{P}\big((|F_3|,|G_3|)=(1/p,(p-1)/p)\big)=\frac{1}{2}\cdot \frac{p\delta}{p\delta+1}.
\end{equation}
Next, for $k=1,\,2,\,\ldots,\,N-1$ we have
\begin{equation}\label{prob2}
\begin{split}
&\mathbb{P}\big((|F_\infty|,|G_\infty|)=((1+2\delta)^k/p,(1+2\delta)^k\cdot (p-1)/p)\big)\\
&=\mathbb{P}\big((|F_{2k+2}|,|G_{2k+2}|)=((1+2\delta)^k/p,(1+2\delta)^k\cdot (p-1)/p)\big)\\
&\quad +\mathbb{P}\big((|F_{2k+3}|,|G_{2k+3}|)=((1+2\delta)^k/p,(1+2\delta)^k\cdot (p-1)/p)\big)\\
&=\frac{1}{2}\left[\frac{1-(p-2)\delta}{(p\delta+1)(1+2\delta)}\right]^{k-1}\frac{p\delta}{(p\delta+1)(1+2\delta)}\cdot\left\{1+\frac{1-(p-2)\delta}{1+p\delta}\right\}.
\end{split}
\end{equation}
Let us explain the latter equality more precisely, focusing on the term $$\mathbb{P}\big((|F_{2k+2}|,|G_{2k+2}|)=((1+2\delta)^k/p,(1+2\delta)^k\cdot (p-1)/p)\big).$$ The event occurs if and only if $F_0=F_2=F_4=\ldots=F_{2k}=0$, $|F_1|=1/2$, $|F_{2n+1}|=\delta(1+2\delta)^{n-1}$ for $n=1,\,2,\,\ldots,\,k$ and $|F_{2k+2}|=(1+2\delta)^k/p$. Directly from the conditions (i)-(iv) above, we see that $\mathbb{P}(|F_2|=0)=1/2$ and
$$ \mathbb{P}(|F_{2n+2}|=0\big||F_{2n}|=0)=\frac{1-(p-2)\delta}{(p\delta+1)(1+2\delta)}.$$
Since
$$ \mathbb{P}(|F_{2k+2}|=(1+2\delta)^k/p||F_{2k}|=0)=\frac{1}{p\delta+1}\cdot \frac{p\delta}{1+2\delta},$$
we get
\begin{align*}
& \mathbb{P}\big((|F_{2k+2}|,|G_{2k+2}|)=((1+2\delta)^k/p,(1+2\delta)^k\cdot (p-1)/p)\big)\\
&=\frac{1}{2}\left[\frac{1-(p-2)\delta}{(p\delta+1)(1+2\delta)}\right]^{k-1}\frac{p\delta}{(p\delta+1)(1+2\delta)}
\end{align*}
A similar calculation shows that
\begin{align*}
 \mathbb{P}\big((|F_{2k+3}|,|G_{2k+3}|)&=((1+2\delta)^k/p,(1+2\delta)^k\cdot (p-1)/p)\big)\\
&=\frac{1}{2}\left[\frac{1-(p-2)\delta}{(p\delta+1)(1+2\delta)}\right]^{k}\frac{p\delta}{p\delta+1}
\end{align*}
and \eqref{prob2} follows.

The last possibility is for $(|F_\infty|,|G_\infty|)$ to reach the state $(0,K)$. An analogous analysis to that above yields
\begin{equation}\label{prob3}
 \mathbb{P}\left((|F_\infty|,|G_\infty|)=(0,K)\right)=\frac{1}{2}\left[\frac{1-(p-2)\delta}{(p\delta+1)(1+2\delta)}\right]^N.
\end{equation}

\emph{Second stage.} Now we modify slightly the martingale $(F,G)$ if it terminates on the lines $y=\pm(p-1)x$. Namely, if $(F,G)$ reached the final value $(y,(p-1)y)$ at some step, then it waits for a time unit, and then goes to $(y-\eta y,(p-1)y+\eta y)$ or $(y+\eta y,(p-1)y-\eta y)$. Similarly, 
when $(F,G)$ reaches the point $(-y,(p-1)y)$ at a certain point in time  then it stays there for a unit of time, and then goes to $(-(y-\eta y),(p-1)y+\eta y)$ or to $(-(y+\eta y),(p-1)y-\eta y)$.

The reason why  the martingale ``waits'' for a unit of time is to preserve the property that $G$ is  the transform of $F$ by the sequence $\{(-1)^n\}_{n=0}^\infty$. The above modification does not affect the probabilities \eqref{prob0} and \eqref{prob3}. On the other hand, the conditions \eqref{prob1} and \eqref{prob2} do change since  the probability of getting to the point of the form $((1+2\delta)^k/p,(1+2\delta)^k\cdot (p-1)/p)$ is split into two halves, corresponding to the new final points
$$ ((1+2\delta)^k/p\cdot (1\pm \eta),(1+2\delta)^k\cdot (p-1)/p \cdot (1\mp \eta/(p-1))).$$

Having completed the construction, we analyze $||F_\infty||_p$, $||G_\infty||_p$ and $\big|\big||G_\infty|-(p-1)|F_\infty|\big|\big|_p$. Denoting the probability in \eqref{prob2} by $p_{k,\delta}$, we get
\begin{align*}
 ||F_\infty||_p^p&=\frac{1}{2}+\frac{1}{2}\cdot \frac{p\delta}{p\delta+1}\left[\frac{(1-\eta)^p+(1+\eta)^p}{2p^p}\right]\\
&\quad +\sum_{k=1}^{N-1} p_{k,\delta}\cdot \left(\frac{(1+2\delta)^k}{p}\right)^p\cdot \frac{(1-\eta)^p+(1+\eta)^p}{2}
\end{align*}
and, omitting the event studied in \eqref{prob1}, we have  the following lower bound for $G$:
\begin{align*}
 ||G_\infty||_p^p&\geq \sum_{k=1}^{N-1} p_{k,\delta}\cdot \left(\frac{(1+2\delta)^k(p-1)}{p}\right)^p\cdot \frac{(1-\eta/(p-1))^p+(1+\eta/(p-1))^p}{2}\\
 &\quad +K^p\cdot \frac{1}{2}\left[\frac{1-(p-2)\delta}{(p\delta+1)(1+2\delta)}\right]^N.
\end{align*}
Concerning $\big|\big||G_\infty|-(p-1)|F_\infty|\big|\big|_p$, we will exploit later two lower bounds for this expression. The first inequality is trivial: just look at the event in \eqref{prob0} to obtain 
\begin{equation}\label{dif0}
 \big|\big||G_\infty|-(p-1)|F_\infty|\big|\big|_p^p\geq \frac{1}{2}(p-1)^p.
\end{equation}
To get the second bound, note that on the set where 
$$ (|F_\infty|,|G_\infty|)=((1+2\delta)^k/p\cdot (1\pm \eta),(1+2\delta)^k\cdot (p-1)/p \cdot (1\mp \eta/(p-1))),$$
we have $|G_\infty|-(p-1)|F_\infty|=\mp (1+2\delta)^k\eta.$ Therefore,
\begin{equation}\label{dif1}
 \big|\big||G_\infty|-(p-1)|F_\infty|\big|\big|_p^p\geq\sum_{k=1}^{N-1} p_{k,\delta}(1+2\delta)^{kp}\eta^p.
\end{equation}

To simplify the later calculations, let us carry out a limiting procedure, by sending $N$ to infinity (but keeping $K$ fixed). Then $\delta$ converges to $0$; to see how the above sums involving $p_{k,\delta}$ behave, observe that the ratio of these geometric sums is given by 
$$ \left[\frac{1-(p-2)\delta}{(p\delta+1)(1+2\delta)}\right](1+2\delta)^{p},$$
which is of order $1+o(\delta)$ as $\delta\to 0$. Consequently (recall that $\lim_{N\to \infty} 2N\delta=\log K$),
\begin{equation}\label{asymptF}
 ||F_\infty||_p^p \xrightarrow{N\to \infty} \frac{1}{2}+\frac{\log K}{2p^{p-1}}\cdot \frac{(1-\eta)^p+(1+\eta)^p}{2}.
\end{equation}
Similarly,
\begin{equation}\label{asymptG}
||G_\infty||_p^p \xrightarrow{N\to \infty}\frac{1}{2}+\frac{(p-1)^p\log K}{2p^{p-1}}\cdot \frac{(1-\eta/(p-1))^p+(1+\eta/(p-1))^p}{2}
\end{equation}
and
\begin{equation}\label{asymptD}
 \liminf_{N\to\infty}\big|\big||G_\infty|-(p-1)|F_\infty|\big|\big|_p^p\geq \frac{p\eta^p\log K}{2}.
\end{equation}

\subsection{Proof of Theorem \ref{mainthm} for $1<p<2$}\label{p<2no}

We will need the following fact.

\begin{lemma}\label{tech2}
For any $1<p\leq 2$ we have 
$$ \frac{p^{2-p}(p-1)^{p-1}(2-p)}{2}\geq 1-p\left(1-\frac{1}{p}\right)^{p-1}.$$
\end{lemma}
\begin{proof}
The claim is equivalent to $$p\left(1-\frac{1}{p}\right)^{p-1}\left(2-\frac{p}{2}\right)\geq 1,$$ or
$$ \log p+(p-1)\log\left(1-\frac{1}{p}\right)+\log\left(2-\frac{p}{2}\right)\geq 0.$$
When $p=2$, both sides are equal; therefore, we will be done if the derivative of the left-hand side is nonpositive on  the interval $(1,2)$. This amounts to verifying that
$$ \frac{2}{p}+\log\left(1-\frac{1}{p}\right)-\frac{1}{4-p}\leq 0.$$
However, one easily checks that for any $x\in (-1,-1/2)$ we have  $-2x+\log(1+x)\leq 1+\log(1/2)$ (the left-hand side is increasing as a function of $x\in (-1,-1/2)$ and both sides are equal for $x=-1/2$). Therefore, plugging $x=-1/p$ we get
$$ \frac{2}{p}+\log\left(1-\frac{1}{p}\right)-\frac{1}{4-p}\leq 1+\log\frac{1}{2}-\frac{1}{4-p}\leq 1+\log\frac{1}{2}-\frac{1}{3}<0,$$
which yields the desired assertion.
\end{proof}

In the proof of the inequality \eqref{main<2} we will exploit the following special function $U_p:\mathbb{H}\times \mathbb{H}\to \R$:
$$ U_p(x,y)=p\left(1-\frac{1}{p}\right)^{p-1}((p-1)|y|-|x|)(|x|+|y|)^{p-1}.$$
This function  was introduced by Burkholder in \cite{B1}.  In \cite{W}, Wang  checked that it satisfies all the requirements of Theorem \ref{Wang}. To establish \eqref{main<2}, we will need the following additional inequality.
\begin{lemma}
For any $x,\,y\in \mathbb{H}$ we have
\begin{equation}\label{maj<2}
U_p(x,y)\geq (p-1)^p|y|^p-|x|^p+\left(1-p\left(1-\frac{1}{p}\right)^{p-1}\right)\frac{((p-1)|y|-|x|)^2}{(|x|+|y|)^{2-p}}.
\end{equation}
\end{lemma}
\begin{proof}
By homogeneity, we may assume that $|x|+|y|=1$. Substituting $s:=|y|\in [0,1]$, we transform the inequality into the following equivalent form
$$ -p\left(1-\frac{1}{p}\right)^{p-1}(ps-1)+(p-1)^ps^p-(1-s)^p+\left(1-p\left(1-\frac{1}{p}\right)^{p-1}\right)(ps-1)^2\leq 0.$$
Denoting the left-hand side by $H(s)$, we derive that
$$ H''(s)=p(p-1)^{p+1}s^{p-2}-p(p-1)(1-s)^{p-2}+2p^2\left(1-p\left(1-\frac{1}{p}\right)^{p-1}\right)$$
is a decreasing function of $s$, with $\lim_{s\to 0+}H''(s)=\infty$ and $\lim_{s\to 1-}H''(s)=-\infty$. Since 
$$ H''(1/p)=2p^2\left[\frac{p^{2-p}(p-1)^{p-1}(p-2)}{2}+1-p\left(1-\frac{1}{p}\right)^{p-1}\right]$$
is nonpositive (by Lemma \ref{tech2}), we see that there is a $p_0\in (0,1/p]$ such that $H$ is convex on $(0,p_0)$ and concave on $(p_0,1)$. Since $H(0)=0$ and $H(1/p)=H'(1/p)=0$, the desired result follows.
\end{proof}

 When passing to the Fourier multipliers in Section \ref{multe}, we will also need the following property of $U_p$.

\begin{lemma}
For any $x,\,y,\,k\in \mathbb{H}$ we have
$$ U_p(x,y)+\langle (U_{p})_{y}(x,y), k\rangle\leq U_p(x,y+k).$$
\end{lemma}
\begin{proof}
For fixed $x\in \mathbb{H}$, the function $y\mapsto U_p(x,y)$ is of class $C^1$, so it suffices to show that the function $H=H_{x,y,k}:\R\to \R$ given by $ H(t)=U_p(x,y+tk)$ is convex. To this end, it is enough to check that $H''(t)\geq  0$ for all $t$ such that the derivative exists; furthermore, since $H_{x,y,k}(u+v)=H_{x,y+uk,k}(v)$, it suffices to verify the inequality $H''(t)\geq  0$ for $t=0$. A direct computation reveals that
$$ H_{x,y,k}'(t)=p(p-1)(\langle y,k\rangle +t|k|^2)(|x|+|y+tk|)^{p-2}$$
and, when $y\neq 0$,
\begin{align*}
 H_{x,y,k}''(0)&=p\left(1-\frac{1}{p}\right)^{p-1}\cdot p(p-1)(|x|+|y|)^{p-3}|x||k|^2\\
&\quad +p\left(1-\frac{1}{p}\right)^{p-1}\cdot p(p-1)(|x|+|y|)^{p-3}|y|\left(|k|^2+(p-2)\langle y/|y|,k\rangle^2\right).
\end{align*}
However, both summands on the right are nonnegative.  This is clear for the first term, while for the second we simply note that $|p-2|<1$ and $\langle y/|y|,k\rangle^2\leq|k|^2$.
\end{proof}

\begin{proof}[Proof of \eqref{main<2}]\label{Sec<2}
Fix $t>0$ and a pair $X$, $Y$ as in the statement. By Theorem \ref{Wang}, there is a nondecreasing sequence $(\tau_n)_{n\geq 0}$ of stopping times converging to infinity such that for each $n$, $\E U_p(X_{\tau_n\wedge t},Y_{\tau_n\wedge t})\leq 0$. Consequently, by \eqref{maj<2},
\begin{align*}
 \left(1-p\left(1-\frac{1}{p}\right)^{p-1}\right)\E \frac{((p-1)|Y_{\tau_n\wedge t}|-|X_{\tau_n\wedge t}|)^2}{(|X_{\tau_n\wedge t}|+|Y_{\tau_n\wedge t}|)^{2-p}}+(p-1)^p\E |Y_{\tau_n\wedge t}|^p&\leq \E |X_{\tau_n\wedge t}|^p\\
 &\leq ||X||_p^p.
\end{align*}
Letting $n\to \infty$ and then $t\to \infty$ we obtain, by Fatou's lemma,
\begin{align*}
 \left(1-p\left(1-\frac{1}{p}\right)^{p-1}\right)\E \frac{((p-1)|Y_\infty|-|X_\infty|)^2}{(|X_\infty|+|Y_\infty|)^{2-p}}&\leq ||X||_p^p-(p-1)^p||Y||_p^p\\
 &\leq \left(1-\left(1-(p-1)\e\right)^p\right)||X||_p^p\\
 &\leq p(p-1)\e||X||_p^p.
\end{align*}
Combining this with H\"older inequality and Burkholder's estimate \eqref{burkin}, we see that
\begin{align*}
||(p-1)|Y_\infty|-|X_\infty|||_p&\leq \left(\E \frac{((p-1)|Y_\infty|-|X_\infty|)^2}{(|X_\infty|+|Y_\infty|)^{2-p}}\right)^{1/2}|||X_\infty|+|Y_\infty|||_p^{1-p/2}\\
&\leq \left(\frac{p(p-1)\e}{1-p\left(1-\frac{1}{p}\right)^{p-1}}\right)^{1/2} ||X||_p^{p/2}\cdot \left(\frac{p}{p-1}||X||_p\right)^{1-p/2}. 
\end{align*}
This is precisely \eqref{main<2}.
\end{proof}

\begin{proof}[Sharpness]
We will now show that the exponent $1/2$ in the factor $\e^{1/2}$ cannot be decreased and also prove that the constant $c_p$ is of optimal order as $p\to 2$. To this end, fix $p\in (1,2)$, a small $\e>0$ and take the example from \S\ref{example}, with small  $\eta$ and a large $K$, to be chosen later. 
As we have observed above, $F$ is a $\pm 1$-transform of $G$.  Furthermore, if $N$ is large enough, then
\begin{align*}
 &||F_\infty||_p^p-\left(\frac{1}{p-1}-\e\right)^p||G_\infty||_p^p\\
 &\geq \frac{\log K}{2p^{p-1}}\frac{(1+\eta)^p+(1-\eta)^p}{2}-\left(\frac{1}{p-1}-\e\right)^p\cdot \frac{1}{2}\\
&\quad -(1-\e(p-1))^p\frac{\log K}{2p^{p-1}}\frac{(1+\eta/(p-1))^p+(1-\eta/(p-1))^p}{2}.
\end{align*}
Now, for any $\eta$ we have
$$ \frac{1}{2}\left[\left(1+\frac{\eta}{p-1}\right)^p+\left(1-\frac{\eta}{p-1}\right)^p\right]\leq 1+\frac{p}{2(p-1)}\eta^2$$
and, if $\eta$ is sufficiently small,
\begin{equation}\label{ineq}
 \frac{(1+\eta)^p+(1-\eta)^p}{2}\geq 1+\frac{p(p-1)}{2}\eta^2-\alpha_p\eta^2,
\end{equation}
where $\alpha_p=p(2-p)/(2(p-1))$. 
For such $\eta$ we can write
\begin{align*}
 &||F_\infty||_p^p-\left(\frac{1}{p-1}-\e\right)^p||G_\infty||_p^p\\
 &\geq -\frac{1}{2}\left(\frac{1}{p-1}-\e\right)^p\\
&\quad +\frac{\log K}{2p^{p-1}}\left[\big(1-(1-(p-1)\e)^p\big)\left(1+\frac{p}{2(p-1)}\eta^2\right)-\frac{p(p+1)(2-p)\eta^2}{2(p-1)}\right].
\end{align*}
Now, for sufficiently small $\e$ we have $1-(1-(p-1)\e)^p\geq (p-1)\e$; taking $\eta=(p-1)\sqrt{\e/(2-p)}/2$ (and decreasing $\e$ if necessary, so that \eqref{ineq} holds) we see that the expression in the square brackets above is not smaller than
$$ (p-1)\e-\frac{p(p-1)(p+1)\e}{8}\geq \frac{(p-1)\e}{4}>0.$$
Therefore, for sufficiently large $K$ we have $||F_\infty||_p^p>\left(\frac{1}{p-1}-\e\right)^p||G_\infty||_p^p$. On the other hand, \eqref{asymptD} implies that for sufficiently large $N$,
$$ \big|\big||G_\infty|-(p-1)|F_\infty|\big|\big|_p^p\geq \left(\frac{p\eta}{2(p-1)}\right)^p \cdot \frac{(p-1)^p\log K}{2p^{p-1}}2^{p}> \left(\frac{p\eta}{2(p-1)}\right)^p ||G_\infty||_p^p, $$
provided $K$ is sufficiently large, so that $||G_\infty||_p^p< \frac{(p-1)^p\log K}{2p^{p-1}}\cdot 2^{p}$. In other words, we have
$$ \big|\big||F_\infty|-(p-1)^{-1}|G_\infty|\big|\big|_p> \frac{p}{4(p-1)}\sqrt{\frac{\e}{2-p}}||G_\infty||_p.$$
This implies the aforementioned optimality of the constants.
\end{proof}

\subsection{Proof of Theorem \ref{mainthm} for $2<p<\infty$}\label{p>2no} Here the reasoning will be slightly longer. We start with the following string of elementary inequalities.   

\begin{lemma}\label{tech}
Let $p\geq 2$. Then
\begin{equation}\label{te1}
 \left(1-\frac{1}{p}\right)^{p-1}\geq \frac{2}{p+2},
\end{equation}
\begin{equation}\label{te2}
 p\left(1-\frac{1}{p}\right)^{p-1}\geq 1+\frac{p-2}{p-1}\left(\frac{1}{2}-\frac{1}{e}\right),
\end{equation}
\begin{equation}\label{te3}
\left(1-\frac{1}{p}\right)^{p-1}\leq 1/2,
\end{equation}
\begin{equation}\label{te4}
\left(1-\frac{1}{p-1}\right)^{p-1}\leq \frac{1}{e}
\end{equation}
and
\begin{equation}\label{te5}
1-\left(1-\frac{1}{p}\right)^{p-2}\geq \frac{p-2}{2(p-1)}.
\end{equation}
\end{lemma}
\begin{proof}
The first inequality is equivalent to 
$$ H(p):=\log(p+2)-\log 2-(p-1)\log\frac{p}{p-1}\geq 0.$$
We have $H'(p)=(p+2)^{-1}+p^{-1}-\log(p/(p-1))$ and
$$ H''(p)=-\frac{1}{(p+2)^2}-\frac{1}{p^2}+\frac{1}{p(p-1)}=\frac{(p+2)^2-p^2(p-1)}{p^2(p-1)(p+2)}.$$
But $I(p)=(p+2)^2-p^2(p-1)$, the numerator of $H''(p)$, is a decreasing function of $p\in [2,\infty)$ and  $I'(p)=-3p^2+4p+4\leq 0$.  Furthermore, we have $I(2)=12$ and $\lim_{p\to\infty}I(p)=-\infty$. Consequently, there is $p_0\in (2,\infty)$ such that $H$ is convex on $[2,p_0]$ and concave on $[p_0,\infty)$. It suffices to note that $H(2)=0$,  $\lim_{p\downarrow 2}H'(p)=\frac{3}{4}-\log 2\geq 0$ and $\lim_{p\to\infty} H(p)=\infty$, and the first estimate follows. The inequality \eqref{te2} is an immediate consequence of \eqref{te1}, since
$$ p\left(1-\frac{1}{p}\right)^{p-1}\geq \frac{2p}{p+2}\geq 1+\frac{p-2}{p-1}\left(\frac{1}{2}-\frac{1}{e}\right),$$
where the latter estimate is equivalent to $1/2+1/e\geq 3/(p+2)$, which is obviously satisfied. 

The inequalities \eqref{te3} and \eqref{te4} follow from a straightforward differentiation, together with the  elementary bound for the logarithmic function: $\frac{x}{x+1}\leq \log(1+x)\leq x$ for $x>-1$. 

Finally, to show \eqref{te5}, note that \eqref{te3} can be rewritten as
$$ 1-\left(1-\frac{1}{p}\right)^{p-2}\geq \left(1-\frac{1}{p}\right)^{p-2}-\frac{1}{p-1}.$$
So, if \eqref{te5} were not true, this would imply that 
$$ \left(1-\frac{1}{p}\right)^{p-2}-\frac{1}{p-1}<\frac{p-2}{2(p-1)},$$
 or equivalently that 
$$ \left(1-\frac{1}{p}\right)^{p-2}<\frac{p-2}{2(p-1)}+\frac{1}{p-1},$$
and this, in turn, would give that 
$$ 1-\left(1-\frac{1}{p}\right)^{p-2}>1-\frac{1}{p-1}-\frac{p-2}{2(p-1)}=\frac{p-2}{2(p-1)},$$
i.e., \eqref{te5}: 
a contradiction.
\end{proof}

As in the case $p<2$, the proof of \eqref{main>2} is based on properties of a certain special function $U_p:\mathbb{H}\times \mathbb{H}\to \R$. Let
 
\begin{equation}\label{U-p}
U_p(x,y)=\begin{cases}
\displaystyle p\left(1-\frac{1}{p}\right)^{p-1}(|y|-(p-1)|x|)(|x|+|y|)^{p-1} &\mbox{if }|y|\geq (p-2)|x|,\\
\displaystyle -\frac{(p-1)^{2p-2}}{p^{p-2}}|x|^p & \mbox{if }|y|<(p-2)|x|.
\end{cases}
\end{equation}
Before we proceed, let us stress here that this function \emph{is not} the function used  by Burkholder (or Wang) in the proof of \eqref{burkin}. To the best of our knowledge, in the literature one can find two proofs of this $L^p$-estimate. One exploits the function by the formula
$$ U_p^{(1)}(x,y)=p\left(1-\frac{1}{p}\right)^{p-1}(|y|-(p-1)|x|)(|x|+|y|)^{p-1},$$
while the other proof uses
$$ U_p^{(2)}(x,y)=\begin{cases}
\displaystyle p\left(1-\frac{1}{p}\right)^{p-1}(|y|-(p-1)|x|)(|x|+|y|)^{p-1} &\mbox{if }|y|\geq (p-1)|x|,\\
\displaystyle \frac{p-1}{p}\left(|y|^p-(p-1)^p|x|^p\right) & \mbox{if }|y|<(p-1)|x|.
\end{cases}$$

Both of these functions are not sufficient for our purposes. As we will see in Section \ref{multe} below, we will require that the function $U_p$ has the following property: for any $x$,  $U_p(x,\cdot)$ is convex. This condition is not satisfied by $U_p^{(1)}$. On the other hand, $U_p^{(2)}$ does not enjoy the appropriate majorization condition (see \eqref{maj>2} below). This forces us to, in a sense, ``mediate'' between $U_p^{(1)}$ and $U_p^{(2)}$, which has led us to the function $U_p$ above.

Let us study the properties of this object.

\begin{lemma}
The function $U_p$ satisfies the assumptions of Theorem \ref{Wang}.
\end{lemma}
\begin{proof}
It is straightforward to check the local boundedness and regularity (there are two sets $S_i$: $S_1=\{(x,y)\in \mathbb{H}\times \mathbb{H}:|y|>(p-2)|x|\}$ and $S_2=\{(x,y)\in \mathbb{H}\times \mathbb{H}:|y|<(p-2)|x|\}$). The only nontrivial assumption is the inequality \eqref{conv}. However, on $S_1$ this estimate is contained in \cite{B2}, while on $S_2$ it is trivial since  the left-hand side equals
$$ -\frac{(p-1)^{2p-2}}{p^{p-2}}\cdot \big(p(p-2)|x|^{p-4}\langle x,h\rangle^2+p|x|^{p-2}|h|^2\big),$$
so the estimate holds with $c_2(x,y)=-p^{3-p}(p-1)^{2p-2}|x|^{p-2}.$ 
\end{proof}

The function $U_p$ enjoys the following majorization property. 

\begin{lemma}
For any $x,\,y\in \mathbb{H}$ we have
\begin{equation}\label{maj>2}
U_p(x,y)\geq |y|^p-(p-1)^p|x|^p+\alpha_p\big||y|-(p-1)|x|\big|^p,
\end{equation}
where
\begin{equation}\label{alpha-p} 
\alpha_p=\frac{p-2}{p-1}\left(\frac{1}{2}-\frac{1}{e}\right).
\end{equation}
\end{lemma}
\begin{proof}
By homogeneity, we may and do assume that $|x|+|y|=1$. Then, substituting $s:=|x|$, we see that the \eqref{maj>2} becomes
\begin{equation}\label{case1}
 -p\left(1-\frac{1}{p}\right)^{p-1}(1-ps)+(1-s)^p-(p-1)^ps^p+\alpha_p|1-ps|^p\leq 0,
\end{equation}
when $s\in [0,1/(p-1)]$, and
\begin{equation}\label{case2}
 -(p-1)^p\left[1-\left(1-\frac{1}{p}\right)^{p-2}\right]s^p+(1-s)^p+\alpha_p (ps-1)^p\leq 0
\end{equation}
if $s\in [1/(p-1),1]$. Denote the left-hand sides of \eqref{case1} and \eqref{case2} by $F(s)$. We easily check that if $s\in (0,1/(p-1))$, then
$$ F''(s)=p(p-1)\bigg[(1-s)^{p-2}-(p-1)^ps^{p-2}+p^2\alpha_p|1-ps|^{p-2}\bigg]=:p(p-1)G(s).$$
Let us analyze the sign of $G$. First, note that $G$ is nonpositive on $[1/p,1/(p-1)]$.  This  is due to the inequality
\begin{align*}
 &(1-s)^{p-2}-(p-1)^ps^{p-2}\\
&=\big[(1-s)^{p-2}-(p-1)^{p-2}s^{p-2}\big]+p(2-p)(p-1)^{p-2}s^{p-2}\\
 &<p(2-p)[(p-1)s]^{p-2}\\
&\leq p(2-p)(ps-1)^{p-2}\leq -p^2\alpha_p(ps-1)^{p-2}.
\end{align*}
To see what happens on the interval $[0,1/p]$, observe that $G$ is decreasing there since it is  the sum of three terms with this property. Furthermore, we have $G(0)=1+p^2(p-2)/(p+2)>0$ and $G(1/p)=p(2-p)(1-1/p)^{p-2}<0$, so there exists $s_0\in (0,1/p)$ such that $G\geq 0$ on $[0,s_0]$ and $G\leq 0$ on $[s_0,1/p]$. Therefore, we have shown the existence of an $s_0\in (0,1/p)$ such that $F$ is convex on $[0,s_0]$ and concave on $[s_0,1/(p-1)]$. Combining this with the equalities $F(1/p)=F'(1/p)=0$ and the estimate $F(0)\leq 0$, which is equivalent to the second inequality of of Lemma \eqref{tech}, we see that the claim will follow once we have shown \eqref{case2} for $s\in [1/(p-1),1]$. To do this, note that for such $s$ we have $1-s\leq (p-2)s$ and $ps-1\leq (p-1)s$.  Hence it suffices to prove that
\begin{equation}\label{innt}
 1-\left(1-\frac{1}{p}\right)^{p-2}-\left(\frac{p-2}{p-1}\right)^p\geq \alpha_p.
\end{equation}
However, by \eqref{te4} and \eqref{te5} we have
$$ -\left(\frac{p-2}{p-1}\right)^p\geq -\frac{p-2}{p-1}\cdot \frac{1}{e}$$
and
$$ 1-\left(1-\frac{1}{p}\right)^{p-2}\geq \frac{p-2}{p-1}\cdot \frac{1}{2}.$$
Summing these two estimates gives  \eqref{innt}. This completes the proof.
\end{proof}

Finally, as in the case $1<p<2$, we will need the convexity of $U_p$ with respect to the variable $y$.

\begin{lemma}\label{conv>2}
For any $x,\,y,\,k\in\mathbb{H}$ we have
$$ U_p(x,y)+\langle (U_{p})_{y}(x,y),k\rangle \leq U_p(x,y+k).$$
\end{lemma}
\begin{proof}
We argue as in the case $p<2$ and consider the function $H(t)=U_p(x+t,y+tk)$. Then $H$ is of class $C^1$ on $\R$ and it is enough to check that $H''(0)\geq 0$,  provided $|y|\neq (p-2)|x|$. If $|y|<(p-2)|x|$, then $H''(0)=0$ since,  if the reverse inequality holds, then $ H''(0)=p\left(1-\frac{1}{p}\right)^{p-1}\cdot(I_1+I_2)$, where
\begin{align*}
I_1&=p(p-2)(|x|+|y|)^{p-2}|x||y|^{-3}\langle y,k\rangle^2,\\
I_2&=p(|x|+|y|)^{p-3}\big((p-2)|y|^{-2}\langle y,k\rangle^2+(|x|/|y|+1)|k|^2\big)[|y|-(p-2)|x|]
\end{align*}
are both nonnegative. This gives the assertion.
\end{proof}

\begin{proof}[Proof of \eqref{main>2}]
Take a pair $X$, $Y$ as in the statement and fix $t\geq 0$. By Theorem \ref{Wang}, there is a nondecreasing sequence $(\tau_n)_{n\geq 0}$ of stopping times converging to infinity such that for each $n$, $\E U_p(X_{\tau_n\wedge t},Y_{\tau_n\wedge t})\leq 0$. We argue as in the case $1<p<2$.  Combining this inequality with \eqref{maj>2} and letting $n$ and $t$ go to infinity,  we get
\begin{align*}
 \frac{p-2}{p-1}\left(\frac{1}{2}-\frac{1}{e}\right)\big|\big||Y_\infty|-(p-1)|X_\infty|\big|\big|_p^p&\leq (p-1)^p||X||_p^p-||Y||_p^p\\
&\leq \big[(p-1)^p-(p-1-\e)^p\big]||X||_p^p\\
 &\leq p(p-1)^{p-1}\e ||X||_p^p.
\end{align*}
This is precisely the inequality \eqref{main>2}.
\end{proof}

\begin{proof}[Sharpness]
Now we will prove that the exponent $1/p$ in the factor $\e^{1/p}$ is optimal.  We will also obtain, using the same example, that the constant $c_p$ has the right order $O(p)$, as $p\to \infty$. 
To this end, consider the example from \S\ref{example}, with $\eta=0$ and some $K$ to be chosen in a moment. Pick a small positive $\e$. By \eqref{asymptF} and \eqref{asymptG}, if $N$ is sufficiently large, then the condition $||G_\infty||_p\geq (p-1-\e)||F_\infty||_p$ is implied by
$$ \frac{(p-1)^p\log K}{2p^{p-1}}=(p-1-\e)^p\left(\frac{1}{2}+\frac{\log K}{2p^{p-1}}\right).$$
This gives the condition on $K$ should be such that 
$$ \frac{\log K}{p^{p-1}}=\left[\left(\frac{p-1}{p-1-\e}\right)^p-1\right]^{-1}< \frac{1}{\e}.$$
But the latter inequality is equivalent to the elementary bound  $\left(\frac{p-1}{p-1-\e}\right)^p> 1+\e$. Hence, by \eqref{dif0},
$$ \big|\big||G_\infty|-(p-1)|F_\infty|\big|\big|_p^p\geq \frac{1}{2}(p-1)^p>\frac{(p-1)^p\e}{2}\cdot \frac{\log K}{p^{p-1}}>\frac{(p-1)^p\e}{2}\cdot ||F_\infty||_p^p, $$
provided $K$ is large enough.

Finally, let us study the order of $c_p$ as $p\downarrow 2$. First, note that for $p$ sufficiently close to $2$ we have
\begin{equation}\label{auxill}
 \frac{2(p-1)^{p-1}(p-2)}{p^{p-2}}>\frac{(p-1)^p-1}{2}.
\end{equation}
Indeed, if we divide both sides  by $p-2$ and let $p\downarrow 2$, then the left-hand side converges to $2$, while the right hand-side converges to $1$. Now, fix $p>2$ such that \eqref{auxill} holds.  Pick a small positive $\e<8(p-2)$ and consider the example of \S\ref{example} with $\eta=0$ and $K=\exp(8(p-2)/\e)$. If $N$ is sufficiently large, then, by \eqref{asymptF} and \eqref{asymptG},
\begin{align*}
 &||G_\infty||_p^p-(p-1-\e)^p||F_\infty||_p^p\\
&\geq \frac{(p-1)^p\log K}{2p^{p-1}}+\frac{1}{2}-(p-1-\e/2)^p\left(\frac{\log K}{2p^{p-1}}+\frac{1}{2}\right)\\
 &=\frac{((p-1)^p-(p-1-\e/2)^p)\cdot 4(p-2)}{p^{p-1}\e}-\frac{(p-1-\e/2)^p-1}{2}.
\end{align*}
If $\e$ is sufficiently small, then the above expression is nonnegative; in the limit $\e\to 0$ this is guaranteed by \eqref{auxill}. It remains to note that for such small $\e$ (and for sufficiently large $N$) we have
$$ \frac{\big|\big||G_\infty|-(p-1)|F_\infty|\big|\big|_p^p}{||F_\infty||_p^p}\geq \frac{1/2}{1+p^{1-p}\log K}=\frac{\e}{p-2}\cdot \left(\frac{2\e}{p-2}+16p^{1-p}\right)^{-1}.$$
This proves that $c_p$ is of the optimal order $O((p-2)^{-1/p}),$ as $p\downarrow 2$.
\end{proof}

\subsection{Proof of Theorem \ref{mainthmo} for $1<p<2$}\label{p<2o} We start with some technical facts.

\begin{lemma}
Let
\begin{equation}\label{au0}
 \kappa_p=-\frac{p-1}{8}\tan^{p-2}\frac{\pi}{2p}\cos\frac{\pi}{p}.
 \end{equation} 
Then 
\begin{equation}\label{au1}
\kappa_p\leq -\frac{p(p-1)}{2}\tan^{p-2}\frac{\pi}{2p}\cos\frac{\pi}{p}, 
\end{equation}
\begin{equation}\label{au3}
\kappa_p\leq -\frac{p(p-1)}{2}\cos\frac{\pi}{p}
\end{equation}
and
\begin{equation}\label{au2}
\kappa_p\leq \tan^{p-2}\frac{\pi}{2p}-\sin^{p-3}\frac{\pi}{2p}\cos\frac{\pi}{2p}.
\end{equation}
\end{lemma}
\begin{proof}
The first inequality is trivial, since $\frac{p}{2}\leq \frac{1}{8}$. The second inequality follows at once from the first one since $\tan^{p-2}\frac{\pi}{2p}\leq 1$. The main difficulty lies in proving \eqref{au2}. Substituting the expression for $kappa_p$ and simplifying, the inequality is equivalent to
$$ -\frac{(p-1)}{8}\cos\frac{\pi}{p}\sin\frac{\pi}{2p}\leq \sin\frac{\pi}{2p}-\cos^{p-1}\frac{\pi}{2p}.$$
We consider two cases.

\emph{The case $1<p\leq 3/2$}. We have $-\cos\frac{\pi}{p}\leq 1$ and $\sin\frac{\pi}{2p}\leq 1$, so we will be done if we show that
\begin{equation}\label{ssam1}
 \frac{p-1}{8}\leq \sin\frac{\pi}{2p}-\cos^{p-1}\frac{\pi}{2p}.
\end{equation}
However, as we shall see, we have
\begin{equation}\label{ge1}
 \sin\frac{\pi}{2p}-1+(2-\sqrt{3})(p-1)\geq 0
\end{equation}
and
\begin{equation}\label{ge2}
1-\cos^{p-1}\frac{\pi}{2p}-(2-\sqrt{2})(p-1)\geq 0.
\end{equation}
Adding these two estimates gives \eqref{ssam1}, since $\sqrt{3}-\sqrt{2}>1/8$. To show \eqref{ge1}, denote the left-hand side by $H(p)$. Differentiating  twice gives   that
$$ H''(p)=\frac{\pi}{4p^4}\left(-\pi\sin\frac{\pi}{2p}+4p\cos\frac{\pi}{2p}\right).$$
The expression in the parentheses increases as $p$ increases.  Furthermore, its values at $p=1$ and $p=3/2$ equal $-\pi$ and $-\pi\sqrt{3}/2+3>0$. Therefore, there is $p_0\in (1,3/2)$ such that $H$ is concave on $(1,p_0)$ and convex on $(p_0,3/2)$. Since $H(1)=0$, \eqref{ge1} will be proved if we can   show that $H'(3/2)\leq 0$. But $H'(p)=-\frac{\pi}{2p^2}\cos\frac{\pi}{2p}+2-\sqrt{3}$ and so
$$ H'(3/2)=-\frac{\pi}{9}+2-\sqrt{3}<0,$$
since $\pi/9>1/3$ and $\sqrt{3}>5/3$. To show \eqref{ge2}, note that $\cos\frac{\pi}{2p}\leq \frac{1}{2}$, so
$$ 1-\cos^{p-1}\frac{\pi}{2p}\geq 1-2^{1-p}.$$
The desired bound now follows at once from the concavity of the function $p\mapsto 1-2^{1-p}$.

\emph{The case $3/2< p\leq 2$}. We have $\sin\frac{\pi}{2p}\leq \frac{\sqrt{3}}{2}$ and $p-1\leq 1$, so it is enough to show that
\begin{equation}\label{ssam2}
 -\frac{\sqrt{3}}{16}\cos\frac{\pi}{p}\leq \sin\frac{\pi}{2p}-\cos^{p-1}\frac{\pi}{2p}.
 \end{equation}
As previously, we split the right-hand side into two parts. We will prove that
\begin{equation}\label{ge3}
\sin\frac{\pi}{2p}-\cos\frac{\pi}{2p}\geq -\frac{\cos\frac{\pi}{p}}{\sqrt{2}}
\end{equation}
and
\begin{equation}\label{ge4}
\cos\frac{\pi}{2p}-\cos^{p-1}\frac{\pi}{2p}\geq -\frac{4}{\pi}(1-\sqrt{2})\cos\frac{\pi}{p}.
\end{equation}
Summing these two bounds gives \eqref{ssam2}, since $\frac{1}{\sqrt{2}}-\frac{4}{\pi}(\sqrt{2}-1)\geq \frac{\sqrt{3}}{16}$. To prove \eqref{ge3}, first note that the elementary identity  
$$ \sin\frac{\pi}{2p}-\cos\frac{\pi}{2p}=\frac{-\cos\frac{\pi}{p}}{\sin\frac{\pi}{2p}+\cos\frac{\pi}{2p}}.$$
Combine this with the simple observation that the denominator is not larger than $\sqrt{2}$ proves  \eqref{ge3}. To establish \eqref{ge4}, notice that the function $u\mapsto u-u^{p-1}$ is increasing on $[1/2,1/\sqrt{2}]$: its derivative  $1-(p-1)u^{p-2}\geq 1-(p-1)2^{2-p}\geq 0$ (the inequality  is equivalent to $2^{p-2}\geq p-1$ and follows from the convexity of the function $p\mapsto 2^{p-2}$). Therefore, using $\cos\frac{\pi}{2p}\geq 1/2$, we obtain
\begin{equation}\label{ge5}
 \cos\frac{\pi}{2p}-\cos^{p-1}\frac{\pi}{2p}\geq \frac{1}{2}(1-2^{2-p})\geq (\sqrt{2}-1)(p-2),
\end{equation}
where in the last inequality we exploited the fact that the function $p\mapsto \frac{1}{2}(1-2^{2-p})$ is concave. Now, observe that the function $\xi(p)=-\cos\frac{\pi}{p}$ is convex on $[3/2,2]$. Indeed, we have
$$ \xi''(p)=\left(-\sin\frac{\pi}{p}\cdot \frac{\pi}{p^2}\right)'=\frac{\pi^2}{p^4}\left(\cos\frac{\pi}{p}+\frac{2\sin\frac{\pi}{p}}{\frac{\pi}{p}}\right)\geq \frac{\pi^2}{p^4}\left(-\frac{1}{2}+\frac{2\cdot \frac{\sqrt{3}}{2}}{\frac{2\pi}{3}}\right)>0,$$
where we used the fact that the function $x\mapsto \sin x/x$ is decreasing on $[\pi/2,2\pi/3]$. Consequently, we have 
\begin{equation}\label{boundxi}
\-\cos\frac{\pi}{p}=\xi(p)=\xi(p)-\xi(2)\geq \xi'(2)(p-2)=\frac{\pi}{4}(2-p),
\end{equation}
which combined with \eqref{ge5} gives
$$  \cos\frac{\pi}{2p}-\cos^{p-1}\frac{\pi}{2p}\geq \frac{4}{\pi}(\sqrt{2}-1)\cos\frac{\pi}{p},$$
which is precisely \eqref{ge4} and completes the proof. 
\end{proof}
Let $V_p:\R\times \R\to \R$ be given by
$$ V_p(x,y)=-\beta_pR^p\cos(p\theta),$$
where 
\begin{equation}\label{beta-p}
\beta_p=\sin^{p-1}\frac{\pi}{2p}/\cos\frac{\pi}{2p}
\end{equation}
 and we have used the polar coordinates $|x|=R\cos\theta$ and  $y=R\sin\theta$, $\theta\in [-\pi/2,\pi/2]$. Pichorides \cite{P} showed that $V_p$ is superharmonic on $\R\times \R$. For out purpose we need its convexity and majorization properties. 

\begin{lemma}
For any $x\in \R$, the function $V_p(x,\cdot )$ is convex.
\end{lemma}
\begin{proof}
It suffices to compute that $ (V_{p})_{yy}(x,y)=p(p-1)\beta_pR^{p-2}\cos((2-p)\theta)$ and note that this expression is nonnegative.
\end{proof}
The key property is the following majorization.

\begin{lemma}
For any $x,\,y\in \R$ we have
\begin{equation}\label{majo<2}
V_p(x,y)\geq |y|^p-\tan^p\frac{\pi}{2p}|x|^p+\kappa_p(|y|-\tan\frac{\pi}{2p}|x|)^2(|x|+|y|)^{p-2}.
\end{equation}
\end{lemma}
\begin{proof}
By symmetry we may assume that $y\geq 0$ (in polar coordinates,  $\theta\in [0,\pi/2]$). We will consider two cases.

\emph{Case 1: $\theta\in [0,\pi/(2p)]$} We will show the stronger bound
$$ V_p(x,y)\geq |y|^p-\tan^p\frac{\pi}{2p}|x|^p+\kappa_p(|y|-\tan\frac{\pi}{2p}|x|)^2|x|^{p-2}.$$
In polar coordinates, this is equivalent to
\begin{equation}\label{sam1}
 \beta_p\frac{\cos(p\theta)}{\cos^p\theta}+\tan^p\theta-\tan^p\frac{\pi}{2p}+\frac{\kappa_p\sin^2(\theta-\frac{\pi}{2p})}{\cos^2\frac{\pi}{2p}\cos^2\theta}\leq 0, 
\end{equation}
where $\kappa_p$  is the constant in  \eqref{au0}.  
Denoting the left-hand side by $H_1(\theta)$, we derive that
$$ \cos^{2}(\theta) H_1'(\theta)=\frac{p\beta_p\sin((p-1)\theta)}{\cos^{p-1}\theta}+p\tan^{p-1}\theta+\frac{2\kappa_p\sin\left(\theta-\frac{\pi}{2p}\right)}{\cos\frac{\pi}{2p}\cos\theta}.$$
Again, denote the right-hand side by $H_2(\theta)$ and differentiate to get
$$ \cos^2(\theta) H_2'(\theta)=-\frac{p(p-1)\beta_p\cos((p-2)\theta)}{\cos^{p-2}\theta}+p(p-1)\tan^{p-2}\theta+2\kappa_p.$$
We repeat this process once again. Denoting the right-hand side by $H_3(\theta)$ and computing we find that 
$$ \cos^{p-1}(\theta) H_3'(\theta)=p(p-1)(p-2)\beta_p\sin((p-3)\theta)+p(p-1)(p-2)\sin^{p-3}\theta.$$
Now, a direct differentiation shows that the right-hand side is nondecreasing; it tends to $-\infty$ when $\theta\downarrow 0$, and its value at $\pi/(2p)$ may be nonpositive or positive, depending on $p$. Consequently, $H_3$ either decreases on $[0,\pi/(2p)]$, or it decreases on some subinterval $[0,\theta_0]$, $\theta_0<\pi/(2p)$, and then increases on $[\theta_0,\pi/(2p)]$. However, we have $H_3(0+)=\infty$ and
$$ H_3\left(\frac{\pi}{2p}\right)=p(p-1)\tan^{p-2}\frac{\pi}{2p}\cos\frac{\pi}{p}+2\kappa_p\leq 0,$$
where the inequality follows from \eqref{au1}. 
Hence, the sign of $H_3$ behaves as follows: there is $\theta_1\in [0,\pi/(2p)]$ such that $H_3>0$ on $(0,\theta_1)$ and $H_3<0$ on $(\theta_1,\pi/(2p))$. So, $H_2$ increases on $(0,\theta_1)$ and decreases on $(\theta_1,\pi/(2p))$. But $H_2(0)<0$ and $H_2(\pi/(2p))=0$.  Therefore, there is $\theta_2\in (0,\pi/(2p))$ such that $H_2<0$ on $[0,\theta_2)$ and $H_2>0$ on $(\theta_2,\pi/(2p))$. So, $H_1$ decreases on $(0,\theta_2)$ and increases on $(\theta_2,\pi/(2p))$. Since 
$$ H_1(0)=\beta_p-\tan^p\frac{\pi}{2p}+\kappa_p\tan^2\frac{\pi}{2p}\leq 0,$$
(by \eqref{au2}) and $H_1(\pi/(2p))=0$, we conclude that $H_1$ is nonpositive on $[0,\pi/(2p)]$, which is precisely \eqref{sam1}.

\emph{Case 2: $\theta\in [\pi/(2p),\pi/2]$}. The reasoning is similar to that above. The majorization follows from the stronger estimate
$$ V_p(x,y)\geq |y|^p-\tan^p\frac{\pi}{2p}|x|^p+\kappa_p(|y|-\tan\frac{\pi}{2p}|x|)^2|y|^{p-2},$$
which, in polar coordinates, can be rewritten in the form
$$ \beta_p\frac{\cos(p\theta)}{\sin^p\theta}+1-\tan^p\frac{\pi}{2p}\cot^p\theta+\frac{\kappa_p\sin^2(\theta-\frac{\pi}{2p})}{\cos^2\frac{\pi}{2p}\sin^2\theta}\leq 0.$$
Denote the left-hand side by $H_1(\theta)$ and compute that
$$ \sin^2(\theta) H_1(\theta)=-\frac{p\beta_p\cos((p-1)\theta)}{\sin^{p-1}\theta}+p\tan^p\frac{\pi}{2p}\cot^{p-1}\theta+\frac{2\kappa_p\sin(\theta-\frac{\pi}{2p})\sin\frac{\pi}{2p}}{\cos^2\frac{\pi}{2p}\sin \theta}.$$
Denote the right-hand side by $H_2(\theta)$ and differentiate this function to obtain
$$ \sin^2(\theta) H_2'(\theta)=\frac{p(p-1)\beta_p\cos((p-2)\theta)}{\sin^{p-2}\theta}-p(p-1)\tan^p\frac{\pi}{2p}\cot^{p-2}\theta+2\kappa_p\tan^2\frac{\pi}{2p}.$$
Finally, denote the right-hand side by $H_3(\theta)$ and derive that
$$ \sin^{p-1}(\theta) H_3'(\theta)=-p(p-1)(p-2)\beta_p\cos((p-3)\theta)+p(p-1)(p-2)\tan^p\frac{\pi}{2p}\cos^{p-3}\theta.$$
Obviously, the right-hand side is decreasing.  Furthermore, it is not difficult to check that
its value at $\pi/(2p)$ equals
$$ p(p-1)(p-2)\frac{\sin^p\frac{\pi}{2p}}{\cos\frac{\pi}{2p}}\left[-3+4\sin^2\frac{\pi}{2p}+\cos^{-2}\frac{\pi}{2p}\right],$$
which is nonpositive, since $p(p-1)(p-2)<0$, $4\sin^2\frac{\pi}{2p}\geq 2$ and $\cos^{-2}\frac{\pi}{2p}\geq 1$. Therefore, $H_3$ is decreasing on the interval $[\pi/(2p),\pi/2]$. But
$$ H_3\left(\frac{\pi}{2p}\right)=2p(p-1)\sin^2\frac{\pi}{2p}-p(p-1)\tan^2\frac{\pi}{2p}+2\kappa_p\tan^2\frac{\pi}{2p}\leq 0,$$
where once again we used  \eqref{au3} above. So,  
$H_3$ is actually nonpositive on $(\pi/(2p),\pi/2)$ and hence $H_2$ is decreasing there. But $H_2$ vanishes at $\pi/(2p)$, so $H_2$ is negative on $[\pi/(2p),\pi/2]$, which implies that $H_1$ is decreasing on this interval. Since $H_1$ also vanishes at $\pi/(2p)$, this shows that $H_1\leq 0$ and completes the proof of the lemma.
\end{proof}

\begin{proof}[Proof of \eqref{maino<2}]\label{Sech<2}
Fix $t>0$ and a pair $X$, $Y$ as in the statement. By Theorem \ref{Wang2}, there is a nondecreasing sequence $(\tau_n)_{n\geq 0}$ of stopping times converging to infinity such that for each $n$, $\E V_p(X_{\tau_n\wedge t},Y_{\tau_n\wedge t})\leq 0$. Consequently, by \eqref{majo<2},
\begin{align*}
 \kappa_p\E \frac{(|Y_{\tau_n\wedge t}|-\tan\frac{\pi}{2p}|X_{\tau_n\wedge t}|)^2}{(|X_{\tau_n\wedge t}|+|Y_{\tau_n\wedge t}|)^{2-p}}+\E |Y_{\tau_n\wedge t}|^p&\leq \tan^p\frac{\pi}{2p}\E |X_{\tau_n\wedge t}|^p\leq \tan^p\frac{\pi}{2p}||X||_p^p.
\end{align*}
Letting $n\to \infty$ and then $t\to \infty$ we obtain, by Fatou's lemma,
\begin{align*}
 \kappa_p\E \frac{(|Y_\infty|-\tan\frac{\pi}{2p}|X_\infty|)^2}{(|X_\infty|+|Y_\infty|)^{2-p}}&\leq \tan^p\frac{\pi}{2p}||X||_p^p-||Y||_p^p\\
 &\leq \left(\tan^p\frac{\pi}{2p}-\left(\tan\frac{\pi}{2p}-\e\right)^p\right)||X||_p^p\\
 &\leq p\tan^{p-1}\frac{\pi}{2p}\e||X||_p^p.
\end{align*}
Applying H\"older's inequality and the estimate \eqref{BWin},  
we obtain
\begin{align*}
\left|\left||Y_\infty|-\tan\frac{\pi}{2p}|X_\infty|\right|\right|_p&\leq \left(\E \frac{(|Y_\infty|-\tan\frac{\pi}{2p}|X_\infty|)^2}{(|X_\infty|+|Y_\infty|)^{2-p}}\right)^{1/2}|||X_\infty|+|Y_\infty|||_{p}^{1-p/2}\\
&\leq \left(\frac{p\tan^{p-1}\frac{\pi}{2p}\e}{\kappa_p}\right)^{1/2}  \left(1+\tan\frac{\pi}{2p}\right)^{1-p/2}||X||_p\\
&= \left(-\frac{8p\tan\frac{\pi}{2p}\e}{(p-1)\cos\frac{\pi}{p}}\right)^{1/2}  \left(1+\tan\frac{\pi}{2p}\right)^{1-p/2}||X||_p.
\end{align*}
Now it suffices to apply $-\cos\frac{\pi}{p}\geq \frac{\pi}{4}(2-p)$ (see \eqref{boundxi} above) and the inequality $\tan\frac{\pi}{2p}\leq (p-1)^{-1}$ (which follows from the comparison of the sharp constants in \eqref{burkin} and \eqref{BWin}) to get the claim.
\end{proof}

\begin{proof}[Sharpness]
Now we will show that the exponent $\e^{1/2}$ and the order $O((p-2)^{-1/2})$ as $p\uparrow 2$ are optimal. We start with the observation that for a given $p\in (1,2)$, if $\eta>0$ is sufficiently small, then
\begin{equation}\label{gen10}
 \frac{\sin^p(\frac{\pi}{2p}+\eta)+ \sin^p(\frac{\pi}{2p}-\eta)}{\cos^p(\frac{\pi}{2p}+\eta)+ \cos^p(\frac{\pi}{2p}-\eta)}\geq \tan^p\frac{\pi}{2p}-d_p(p-2)\eta^2,
 \end{equation}
 where $d_p=2p^3/(p-1)^p$. 
To show this, note that
\begin{align*}
 &\sin^p\left(\frac{\pi}{2p}+\eta\right)+ \sin^p\left(\frac{\pi}{2p}-\eta\right)\\
&\qquad =
2\sin^p\frac{\pi}{2p}+\left(p(p-1)\sin^{p-2}\frac{\pi}{2p}\cos^2\frac{\pi}{2p}-p\sin^p\frac{\pi}{2p}\right)\eta^2+o(\eta^2)
\end{align*}
and
\begin{align*}
 &\cos^p\left(\frac{\pi}{2p}+\eta\right)+ \cos^p\left(\frac{\pi}{2p}-\eta\right)\\
&\qquad =
2\cos^p\frac{\pi}{2p}+\left(p(p-1)\cos^{p-2}\frac{\pi}{2p}\sin^2\frac{\pi}{2p}-p\cos^p\frac{\pi}{2p}\right)\eta^2+o(\eta^2),
\end{align*}
which, after some straightforward manipulations, implies
\begin{align*}
 &\frac{\sin^p(\frac{\pi}{2p}+\eta)+ \sin^p(\frac{\pi}{2p}-\eta)}{\cos^p(\frac{\pi}{2p}+\eta)+ \cos^p(\frac{\pi}{2p}-\eta)}- \tan^p\frac{\pi}{2p}\\
 &\qquad \qquad = \left[p(p-1)\left(\tan^{p-2}\frac{\pi}{2p}-\tan^2\frac{\pi}{2p}\right)+p\left(1-\tan^p\frac{\pi}{2p}\right)\right]\eta^2+o(\eta^2)\\
 &\qquad \qquad \geq p^2\left(1-\tan^p\frac{\pi}{2p}\right)\eta^2+o(\eta^2)\\
&\qquad \qquad \geq p^2\left(1-\frac{1}{(p-1)^p}\right)\eta^2+o(\eta^2)\\
&\qquad \qquad =p^2\frac{(p-1)^p-1}{(p-1)^p}\eta^2+o(\eta^2)\geq \frac{p^3(p-2)}{(p-1)^p}\eta^2+o(\eta^2). 
\end{align*}
Hence \eqref{gen10} follows. Now, pick a positive number $\e$ and $\xi<\pi/(2p)$. By \eqref{gen10}, if $\xi$ is sufficiently close to $\pi/(2p)$ and $\e$ is small enough, then 
$$ \frac{\sin^p(\xi+\eta)+ \sin^p(\xi-\eta)}{\cos^p(\xi+\eta)+ \cos^p(\xi-\eta)}\geq \tan^p\left(\frac{\pi}{2p}-\e\right)
$$
provided $p^3(2-p)(p-1)^{-p}\eta^2\leq \e$ (here we have used the inequality $\tan^p(\pi/(2p))\geq \tan^p(\pi/(2p)-\e)+\e$, valid for sufficiently small $\e$). 

Now, consider the angle
$$D=\{(x,y):x>-1,\,-\tan(\xi+\eta)(x+1)\leq y\leq \tan(\xi-\eta)(x+1)\}$$
and let $(X,Y)$ be a two-dimensional Brownian motion starting from the origin, killed upon hitting the boundary of $D$. Since the aperture of $D$ is smaller then $\pi/p$, both $X$ and $Y$ are $L^p$ bounded; furthermore, if $\xi$ is sufficiently close to $\pi/(2p)$, then the $L^p$-norm of $X$ can be made arbitrarily large. For any $a,\,b\in \R$, the function
$$ W_p(x,y)=aR^p\cos p\theta+bR^p\sin p\theta$$
is harmonic in $D$; if $a$, $b$ are chosen such that 
$$ W_p(x,y)=y^p-\tan^p\left(\frac{\pi}{2p}-\e\right)|x|^p$$
for $(x,y)\in \partial D$, then a straightforward use of It\^o's formula implies
$$ ||Y||_p^p-\tan^p\left(\frac{\pi}{2p}-\e\right)||X+1||_p=W_p(1,0)=a.$$
Now, we easily find $a$ and $b$; actually, we only need to study the first of them, equal to
\begin{align*}
 a&=\sin^{-1}2p\xi\bigg\{\sin p(\xi+\eta)\left[\sin^p(\xi-\eta)-\tan^p(\frac{\pi}{2}-\e)]\cos^p(\xi-\eta)\right]\\
 &\qquad \qquad \qquad \qquad +\sin p(\xi-\eta)\left[\sin^p(\xi+\eta)-\tan^p(\frac{\pi}{2}-\e)]\cos^p(\xi+\eta)\right]\bigg\}.
\end{align*}
Take a look at the expression in the parentheses. If $\xi$ were equal to $\pi/(2p)$, then \eqref{gen10} would guarantee the positivity of the expression, with $\eta=(\e/(2-p))^{1/2}$; by continuity, this is also true if $\xi$ is a little less than $\pi/(2p)$. In other words, if $\xi$ and $\eta$ are chosen in such a way, then $||Y||_p\geq \left(\tan\frac{\pi}{2p}-\e\right)||X+1||_p$; by the aforementioned explosion of $L^p$-norms for $\xi\uparrow \pi/(2p)$, we see that
$$ ||Y||_p\geq \left(\tan\frac{\pi}{2p}-2\e\right)||X||_p$$
provided $\xi$ is sufficiently close to $\pi/2p$. Next, by the definition of $D$, $|Y_\infty|=\tan(\xi\pm \eta)|X_\infty+1|$, so 
\begin{align*}
 \big|\big||Y_\infty|-\tan\frac{\pi}{2p}|X_\infty|\big|\big|_p&\geq \big|\big||Y_\infty|-\tan\frac{\pi}{2p}|X_\infty+1|\big|\big|_p-\tan\frac{\pi}{2p}\\
&\geq \eta||X+1||_p-\tan\frac{\pi}{2p}\geq \frac{\eta}{2}||X||_p,
\end{align*}
provided $\xi$ is sufficiently close to $\pi/(2p)$. The desired sharpness follows.
\end{proof}

\subsection{Proof of Theorem \ref{mainthmo} for $2<p<\infty$}\label{p>2o}

We will need the following fact.

\begin{lemma}
Let
$$ \mu_p=\left(1-\frac{\sqrt{2}}{2}\right)\frac{p-2}{p}.$$
Then
\begin{equation}\label{auxx1}
\mu_p\leq 1-\frac{\sin^{p-1}\frac{\pi}{2p}}{\cos\frac{\pi}{2p}\sin^{p-1}\frac{\pi}{2(p-1)}},
\end{equation}
\begin{equation}\label{auxx2}
\mu_p\leq \frac{\cos^{p-2}\frac{\pi}{2p}\sin^{p-2}\frac{\pi}{2p}\cos\frac{\pi}{p}}{\sin^{p-2}\frac{\pi}{2p(p-1)}},
\end{equation}
and
\begin{equation}\label{auxx3}
\mu_p\leq \frac{\cos^{p-1}\frac{\pi}{2p}}{\sin\frac{\pi}{2p}}-1.
\end{equation}
\end{lemma}
\begin{proof}[Proof of \eqref{auxx1}]
First we will prove that for $p\geq 2$ we have the estimate
\begin{equation}\label{gen0}
\frac{\sin \frac{\pi}{2p}}{\sin\frac{\pi}{2(p-1)}}\leq \left(\frac{p-1}{p}\right)^{1/2}.
\end{equation}
To this end, consider the function $\xi(x)=x^{-1/2}\sin x$, $x\in [0,\pi/2]$ (we set $\xi(0)=0$). We easily check that $\xi'(x)=x^{-3/2}\cos x(x-\frac{1}{2}\tan x)$, so there is $x_0\in (\pi/4,\pi/2)$ such that $\xi$ increases on $[0,x_0]$ and decreases on $[x_0,\pi/2]$. Now, take a look at the difference
$$ \left(\frac{\pi}{2p}\right)^{-1/2}\sin\frac{\pi}{2p}-\left(\frac{\pi}{2(p-1)}\right)^{-1/2}\sin\frac{\pi}{2(p-1)}.$$
If $p\geq p_0$, where $\pi/(2(p_0-1))=x_0$, then the difference is nonpositive: this is due to the monotonicity of $\xi$. Now, if we decrease $p$ from $p_0$ to $2$, then the expression $\left(\frac{\pi}{2p}\right)^{-1/2}\sin\frac{\pi}{2p}$ increases, while $\left(\frac{\pi}{2(p-1)}\right)^{-1/2}\sin\frac{\pi}{2(p-1)}$ decreases (again, this follows from the monotonicity of $\xi$ and the fact that $ \pi/(2p)\leq x_0$). It suffices to note that for $p=2$ the difference is zero; this proves that for any $p\geq 2$ the difference is nonpositive, which is equivalent to \eqref{gen0}. This inequality implies
$$ \left(\frac{\sin \frac{\pi}{2p}}{\sin\frac{\pi}{2(p-1)}}\right)^{p-1}\leq \left(\frac{p-1}{p}\right)^{(p-1)/2}\leq \left(\frac{2-1}{2}\right)^{1/2}=2^{-1/2}.$$
Consequently,
\begin{align*}
 1-\frac{\sin^{p-1}\frac{\pi}{2p}}{\cos\frac{\pi}{2p}\sin^{p-1}\frac{\pi}{2(p-1)}}&\geq 1-\frac{2^{-1/2}}{\cos\frac{\pi}{2p}}=\frac{\cos\frac{\pi}{2p}-2^{-1/2}}{\cos\frac{\pi}{2p}}\geq \cos\frac{\pi}{2p}-2^{-1/2}.
\end{align*}
Now, by the concavity of the cosine function on $[0,\pi/4]$, we have  $\cos x-2^{-1/2}\geq (1-4x/\pi)(1-2^{-1/2})$; plugging $x=\pi/(2p)$ and working a little bit, we get \eqref{auxx1}.
\end{proof}

\begin{proof}[Proof of \eqref{auxx2} and \eqref{auxx3}]
First we will show that 
\begin{equation}\label{gen2}
\cos^{p-1}\frac{\pi}{2p}\geq 2^{-1/2}.
\end{equation}
To this end, we will prove that the left-hand side is an increasing function of $p$ (note that for $p=2$ both sides are equal). Plugging $x=1/p$ and taking logarithm, this is equivalent to saying that the function $x\mapsto (x^{-1}-1)\ln\cos\frac{\pi}{2}x$ is increasing on $[0,1/2]$. Differentiating and manipulating a little bit, we obtain the equivalent statement
$$ H(x):=\ln\cos\frac{\pi}{2}x+\frac{\pi}{2}(x-x^2)\tan\frac{\pi}{2}x\geq 0.$$
However,
$$ H'(x)=-\frac{\pi}{2}\tan \frac{\pi}{2}x+\frac{\pi^2}{4}(x-x^2)\cos^{-2}\frac{\pi}{2}x$$
has the same sign as $-\sin \pi x+(1-x)\pi/2$. This expression is positive for $x=0$, decreasing on $[0,1/2]$ and negative for $x=1/2$. Consequently, there is $x_0\in (0,1/2)$ such that $H$ increases on $(0,x_0)$ and decreases on $(x_0,1)$. However, $H(0)=0$ and $H(1/2)=-\frac{1}{2}\ln 2+\frac{\pi}{8}\geq 0$. This shows $H\geq 0$ on $[0,1/2]$ and completes the proof of \eqref{gen2}. This estimate, together with the trivial bound $\sin\frac{\pi}{2p}\geq \sin\frac{\pi}{2p(p-1)}$, implies
$$ \frac{\cos^{p-2}\frac{\pi}{2p}\sin^{p-2}\frac{\pi}{2p}\cos\frac{\pi}{p}}{\sin^{p-2}\frac{\pi}{2p(p-1)}}\geq \frac{2^{-1/2}\cos\frac{\pi}{p}}{\cos\frac{\pi}{2p}}\geq 2^{-1/2}\cos\frac{\pi}{p}\geq 2^{-1/2}\cdot \frac{p-2}{p}\geq \mu_p.$$
Here in the middle we have used the elementary estimate $\cos x\geq 1-\frac{2}{\pi}x$ for $x\in [0,\pi/2]$ (and applied it to $x=\pi/p$). Thus \eqref{auxx2} is established, and \eqref{auxx3} also follows quickly: by \eqref{gen2},
$$ \frac{\cos^{p-1}\frac{\pi}{2p}}{\sin\frac{\pi}{2p}}-1\geq \frac{2^{-1/2}}{\sin\frac{\pi}{2p}}-1\geq 2^{-1/2}-\sin\frac{\pi}{2p}=\sin\frac{\pi}{4}-\sin\frac{\pi}{2p}\geq \frac{\pi}{4\sqrt{2}}\cdot \frac{p-2}{p}\geq \mu_p,$$
where the difference of the sine functions was bounded with the use of mean-value theorem. The proof is complete.
\end{proof}

Consider $V_p:\R\times \R\to \R$ given by
$$ V_p(x,y)=\begin{cases}
\displaystyle \beta_pR^p\cos\left(p\left(\frac{\pi}{2}-\theta\right)\right) & \mbox{if }\theta\geq \frac{\pi}{2}-\frac{\pi}{2(p-1)},\\
-\gamma_p|x|^p & \mbox{if }\theta\leq \frac{\pi}{2}-\frac{\pi}{2(p-1)}, 
\end{cases}$$
where 
$$\beta_p=\frac{\cos^{p-1}\frac{\pi}{2p}}{\sin\frac{\pi}{2p}}$$ and $$ \gamma_p=\frac{\cos^{p-1}\frac{\pi}{2p}}{\sin\frac{\pi}{2p}\sin^{p-1}\frac{\pi}{2(p-1)}}.$$
We have used the polar coordinates: $|x|=R\cos\theta$, $|y|=R\sin\theta$, $R\geq 0$ and $\theta\in [0,\pi/2]$.

\begin{lemma}
The function $V_p$ is superharmonic and for each $x\in \R$, the function $V_p(x,\cdot)$ is convex.
\end{lemma}
\begin{proof}
It is not difficult to check that $V_p$ is of class $C^1$, so it suffices to verify that $\Delta_{y} V_p\leq 0$ and $(V_p)_{yy}\geq 0$, for $\theta<\frac{\pi}{2}-\frac{\pi}{2(p-1)}$ and for $\theta>\frac{\pi}{2}-\frac{\pi}{2(p-1)}$. In the first case the inequalities are trivial. $\Delta_y V_p(x,y)=-p(p-1)\gamma_p|x|^{p-2}<0$ and $(V_p)_{yy}(x,y)=0$; on the set $\theta>\frac{\pi}{2}-\frac{\pi}{2(p-1)}$ the laplacian vanishes and we have
$$ (V_{p})_{yy}(x,y)=p(p-1)\beta_p R^{p-2}\cos\left((p-2)\left(\frac{\pi}{2}-\theta\right)\right)\geq 0.$$
This proves the assertion.
\end{proof}

As in the case $1<p<2$, the main difficulty lies in proving the appropriate majorization condition.

\begin{lemma}
We have
\begin{equation}\label{majo>2}
V_p(x,y)\geq |y|^p-\cot^p\frac{\pi}{2p}|x|^p+\mu_p\left||y|-\cot\frac{\pi}{2p}|x|\right|^p.
\end{equation}
\end{lemma}
\begin{proof}
We consider two cases separately.

\emph{The case $\theta\leq \frac{\pi}{2}-\frac{\pi}{2(p-1)}$}. Here the situation is simple. The majorization can be rewritten in the form
$$ \tan^p\theta+\gamma_p-\cot^p\frac{\pi}{2p}+\frac{\mu_p\cos^p\left(\theta+\frac{\pi}{2p}\right)}{\sin^p\frac{\pi}{2p}\cos^p\theta}\leq 0.$$
Denote the left-hand side by $H_1(\theta)$ and compute that
$$ H_1'(\theta)=\frac{p\sin^{p-1}\theta}{\cos^{p+1}\theta}\left[1-\frac{\mu_p}{\sin^{p-1}\frac{\pi}{2p}}\left(\frac{\cos\left(\theta+\frac{\pi}{2p}\right)}{\sin\theta}\right)^{p-1}\right].$$
Since
$$ \left(\frac{\cos\left(\theta+\frac{\pi}{2p}\right)}{\sin\theta}\right)'=-\frac{\cos\frac{\pi}{2p}}{\sin^2\theta}<0,$$
the expression in the square brackets above is an increasing function of $\theta$; this expression is negative when $\theta$ is close to $0$ and may be positive/nonpositive for $\theta=\frac{\pi}{2}-\frac{\pi}{2(p-1)}$. Consequently, it is enough to check the majorization for $\theta=0$ (and then it holds: see \eqref{auxx1}) and for $\theta=\frac{\pi}{2}-\frac{\pi}{2(p-1)}$ (this will follow from the next case).

\emph{The case $\theta\geq \frac{\pi}{2}-\frac{\pi}{2(p-1)}$}. Here the calculations are more elaborate. We must show that
$$ -\beta_p\frac{\cos(p(\frac{\pi}{2}-\theta))}{\sin^p\theta}+1-\cot^p\frac{\pi}{2p}\cot^p\theta+\frac{\mu_p}{\sin^p\frac{\pi}{2p}}\left|\frac{\cos(\theta+\frac{\pi}{2p})}{\sin\theta}\right|^p\leq 0.$$
Denote the left-hand side by $H_1(\theta)$ and differentiate to obtain
$$ \frac{\sin^{p+1}\theta}{p\cos^{p-1}\theta}H_1'(\theta)=\beta_p\frac{\cos(\frac{p\pi}{2}-(p-1)\theta)}{\cos^{p-1}\theta}+\cot^p\frac{\pi}{2p}\pm\frac{\mu_p\cos\frac{\pi}{2p}}{\sin^p\frac{\pi}{2p}}\left|\frac{\cos(\theta+\frac{\pi}{2p})}{\cos\theta}\right|^{p-1},$$
where $\pm=-\operatorname*{sgn}\cos(\theta+\frac{\pi}{2p})$. Denote the right-hand side by $H_2(\theta)$ and compute that
$$ H_3(\theta)= \frac{\cos^p\theta}{p-1}H_2'(\theta)=\beta_p\sin\left(\frac{p\pi}{2}-(p-2)\theta\right)+\frac{\mu_p\cos\frac{\pi}{2p}}{\sin^{p-1}\frac{\pi}{2p}}\left|\cos\left(\theta+\frac{\pi}{2p}\right)\right|^{p-2}.$$
The function $\theta\mapsto \sin\left(\frac{p\pi}{2}-(p-2)\theta\right)=-\sin\left((p-2)\left(\frac{\pi}{2}-\theta\right)\right)$ is increasing on the interval $[\frac{\pi}{2}-\frac{\pi}{2(p-1)},\frac{\pi}{2}]$, while $\theta\mapsto |\cos(\theta+\frac{\pi}{2p})|$ is decreasing on $[\frac{\pi}{2}-\frac{\pi}{2(p-1)},\frac{\pi}{2}-\frac{\pi}{2p}]$ and increasing on $[\frac{\pi}{2}-\frac{\pi}{2p},\frac{\pi}{2}]$. This implies the following: if $\theta\in [\frac{\pi}{2}-\frac{\pi}{2(p-1)},\frac{\pi}{2}-\frac{\pi}{2p}]$, then $H_3(\theta)$ does not exceed
\begin{align*}
& \beta_p\sin\left(\frac{p\pi}{2}-(p-2)\left(\frac{\pi}{2}-\frac{\pi}{2p}\right)\right)+\frac{\mu_p\cos\frac{\pi}{2p}}{\sin^{p-1}\frac{\pi}{2p}}\left|\cos\left(\frac{\pi}{2}-\frac{\pi}{2(p-1)}+\frac{\pi}{2p}\right)\right|^{p-2}\\
&=-\beta_p\cos\frac{\pi}{p}+\frac{\mu_p\cos\frac{\pi}{2p}}{\sin^{p-1}\frac{\pi}{2p}}\sin^{p-2}\frac{\pi}{2p(p-1)}\leq 0
\end{align*}
(see \eqref{auxx2}). On the other hand, $H_3$ increases on $[\frac{\pi}{2}-\frac{\pi}{2p},\frac{\pi}{2}]$ and $H_3(\pi/2)>0$. Consequently, there is $\theta_0\in [\frac{\pi}{2}-\frac{\pi}{2p},\frac{\pi}{2}]$ such that $H_3$ is negative on $[\frac{\pi}{2}-\frac{\pi}{2(p-1)},\theta_0)$ and positive on $(\theta_0,\frac{\pi}{2}]$. This implies that $H_2$ decreases on the first interval and increases on the second. But $H_2(\frac{\pi}{2}-\frac{\pi}{2p})=0$ and $H_2(\frac{\pi}{2})>0$. This implies that there is $\theta_1\in [\theta_0,\frac{\pi}{2}]$ such that $H_1$ is increasing on $[\frac{\pi}{2}-\frac{\pi}{2(p-1)},\frac{\pi}{2}-\frac{\pi}{2p}]$, decreasing on $[\frac{\pi}{2}-\frac{\pi}{2p},\theta_1]$ and increasing on $[\theta_1,\frac{\pi}{2}]$. Since $H_1(\frac{\pi}{2}-\frac{\pi}{2p})=0$, the desired inequality $H_1\leq 0$ follows from $H_1(\frac{\pi}{2})\leq 0$, which is guaranteed by \eqref{auxx3}. 
\end{proof}

\begin{proof}[Proof of \eqref{maino>2}]
The argument is the same as in the case $p<2$. We omit the details.
\end{proof}

\begin{proof}[Sharpness]
Let $\e>0$ be fixed. Take $\xi=\frac{\pi}{2p}-b_p\e$, where
$$ b_p=\frac{\sin\frac{\pi}{p}}{p(p-2)}.$$
Let $(X,Y)$ be a two-dimensional Brownian motion starting from $(0,0)$ and stopped upon reaching the boundary of the set $D=\{(x,y):y+1\geq \cot\xi|x|\}$. A direct use of Ito's formula, applied to the harmonic functions $R^\alpha \cos \alpha \theta$, shows that
$$ \E |Y_\infty+1|^p=\frac{\cos^p\xi}{\cos p\xi}, \E |Y_\infty+1|^{p-1}=\frac{\cos^{p-1}\xi}{\cos (p-1)\xi}, \E |Y_\infty+1|^{p-2}=\frac{\cos^{p-2}\xi}{\cos (p-2)\xi}$$
and
$$ \E |X_\infty|^p=\frac{\sin^p \xi}{\cos p\xi}.$$
Therefore, using the elementary inequality $y^{p/2}-x^{p/2}\leq \frac{p}{2}y^{p/2-1}(y-x)$, valid for $x,\,y\geq 0$, we compute that
\begin{align*}
 &||Y||_p^p-\left(\cot\frac{\pi}{2p}-\e\right)^p||X||_p^p\\
&=||Y+1||_p^p-\left(\cot\frac{\pi}{2p}-\e\right)^p||X||_p^p+\big(||Y||_p^p-||Y+1||_p^p\big)\\
&\geq \frac{\cos^p\xi}{\cos p\xi}-\left(\cot\frac{\pi}{2p}-\e\right)^p\frac{\sin^p \xi}{\cos p\xi}-\frac{p}{2}\left(\E|Y_\infty+1|^{p/2}(|Y_\infty+1|^2-|Y_\infty|^2\right)\\
&=\frac{\sin^p\xi}{\cos p\xi}\left[\cot^p\xi-\left(\cot\frac{\pi}{2p}-\e\right)^p\right]-\frac{p}{2}\left[\frac{2\cos^{p-1}\xi}{\cos(p-1)\xi}-\frac{\cos^{p-2}\xi}{\cos(p-2)\xi}\right].
\end{align*}
Letting $\e\to 0$, we  see that $\xi\to \pi/(2p)$, so calculating a little bit gives that the above expression converges to
\begin{align*}
&\cos^{p-2}\frac{\pi}{2p}\left[\frac{\sin\frac{\pi}{p}}{b_p}-\frac{4(p-1)\cos^2\frac{\pi}{2p}-p}{2\sin\frac{\pi}{p}}\right]\\
&\qquad =\cos^{p-2}\frac{\pi}{2p}\left[p(p-2)-\frac{4(p-1)\cos^2\frac{\pi}{2p}-p}{2\sin\frac{\pi}{p}}\right].
\end{align*}
However, $\sin\frac{\pi}{p}> 2/p$ and
\begin{align*}
 4(p-1)\cos^2\frac{\pi}{2p}-p&=4\left(\frac{p}{2}-1\right)\cos^2\frac{\pi}{2p}+2p\left(\cos^2\frac{\pi}{2p}-\frac{1}{2}\right)\\
 &\leq 4\left(\frac{p}{2}-1\right)+2p\left(\frac{\pi}{4}-\frac{\pi}{2p}\right)\leq 4(p-2),
\end{align*}
where in the middle we have used the estimate $\cos^2 x-\cos^2\frac{\pi}{4}\leq \frac{\pi}{4}-x$, which follows directly from the mean-value property. Putting all the above facts together we see that if $\e$ is sufficiently small, then $||Y||_p\geq \left(\cot\frac{\pi}{2p}-\e\right)||X||_p$. On the other hand, consider the set $D\cap \R\times [-1/100,1/100]$. With some positive probability (which can be bounded from below by a positive constant $\eta$ not depending on $p$), the process $(X,Y)$ never leaves this set. That is it terminates inside it. If such a situation occurs, then $|X_\infty|\geq \frac{99}{100}\tan \xi$ and $|Y_\infty|\leq 1/100$.  Thus  for small $\e$,
$$ \left| |Y_\infty|-\cot\frac{\pi}{2p}|X_\infty|\right|\geq \frac{99}{100}\tan \xi\cot\frac{\pi}{2p}-\frac{1}{100}\geq 1/2$$
and hence $|||Y_\infty|-\cot\frac{\pi}{2p}|X_\infty|||_p\geq \eta^{1/p}/2$. It remains to note that
$$ ||X||_p=\frac{\sin\xi}{(\cos p\xi)^{1/p}}=\sin \frac{\pi}{2p}\left(\sin \frac{\pi}{p}\cdot \frac{\e}{p-2}\right)^{-1/p}+o(\e).$$
This proves the optimality of the constants.
\end{proof}

\section{Proofs of results for Fourier multipliers}\label{multe}
Here we will show how the martingale inequalities studied in the preceding section yield the corresponding stability results for Fourier multipliers.

\subsection{Proof of inequalities \eqref{mainf<2} and \eqref{mainf>2} in Theorem \ref{mainthmf}} Let us begin by recalling the probabilistic representation of the multipliers from the class \eqref{defm}. We follow here the description  in \cite{BBB} and \cite{BB} and refer the reader to those papers for full details.  Let $\nu$ be a finite, nonzero L\'evy measure on $\R^d$, i.e., a nonnegative Borel measure on $\R^d$ which does not charge the origin and satisfies $\nu(\R^d)<\infty$ and 
$$ \int_{\R^d}\min\{|x|^2,1\}\nu(\mbox{d}x)<\infty.$$
Then for any $s<0$, there is a L\'evy process $(X_{s,t})_{t\in [s,0]}$ with $X_{s,s}\equiv 0$, for which Lemmas \ref{diflema} and \ref{diflema2} below hold true. To state these, we need some notation. For a given complex-valued $f\in L^\infty(\R^d)$, define the corresponding parabolic extension $\mathcal{U}_f$ to $(-\infty,0]\times \R^d$ by
$$ \mathcal{U}_f(s,x)=\E f(x+X_{s,0}).$$
Next, fix $x\in \R^d$, $s<0$ and a complex-valued $\phi\in L^\infty(\R^d)$. We introduce the processes $F=(F^{x,s,f}_t)_{s\leq t\leq 0}$ and $G=(G^{x,s,f,\phi}_t)_{s\leq t\leq 0}$ by
\begin{equation}\label{defFG}
\begin{split}
 F_t&=\mathcal{U}_f(t,x+X_{s,t}),\\
 G_t&=\sum_{s<u\leq t}\big[(F_u-F_{u-}) \cdot \phi(X_{s,u}-X_{s,u-})\big]\\
&\quad -\int_s^t\int_{\R^d}\big[\mathcal{U}_f(v,x+X_{s,v-}+z)-\mathcal{U}_f(v,x+X_{s,v-}) \big]\phi(z)\nu(\mbox{d}z)\mbox{d}v.
\end{split}
\end{equation}
Now, fix $s<0$ and define the operator $\mathcal{S}=\mathcal{S}^{s,\phi,\nu}$ by the bilinear form
\begin{equation}\label{defS}
 \int_{\R^d}\mathcal{S}f(x)g(x)\mbox{d}x=\int_{\R^d}\E \big[G_0^{x,s,f,\phi}g(x+X_{s,0})\big]\mbox{d}x,
\end{equation}
where $f,\,g\in C_0^\infty(\R^d)$. Standard density argument implies that if $1<p<\infty$, then the above identity holds true for all $f\in L^p(\R^d)$. 

We have the following facts, proved in \cite{BBB} and \cite{BB}. 
\begin{lemma}\label{diflema}
For any fixed $x,\,s,\,f,\,\phi$ as above, the processes $F^{x,s,f}$, $G^{x,s,f,\phi}$ are martingales with respect to $(\F_t)_{s\leq t\leq 0}=(\sigma(X_{s,t}:s\leq t))_{s\leq t\leq 0}$. Furthermore, if $||\phi||_\infty\leq 1$, then $G^{x,s,f,\phi}$ is differentially subordinate to $F^{x,s,f}$.
\end{lemma}

The aforementioned representation of Fourier multipliers in terms of L\'evy processes is as follows.

\begin{lemma}\label{diflema2}
Let $1<p<\infty$ and $d\geq 2$. The operator $\mathcal{S}^{s,\phi,\nu}$ is well defined and extends to a bounded operator on $L^p(\R^d)$, which can be expressed as a Fourier multiplier with the symbol
\begin{equation}\label{defMs}
\begin{split}
 &M_{s,\phi,\nu}(\xi)\\
&=\left[1-\exp\left(2s\int_{\R^d}(1-\cos\langle \xi, z\rangle )\nu(\mbox{d}z)\right)\right]
\frac{\int_{\R^d}(1-\cos\langle \xi, z\rangle)\phi(z)\nu(\mbox{d}z)}{\int_{\R^d}(1-\cos\langle \xi, z\rangle)\nu(\mbox{d}z)}
\end{split}
\end{equation}
if $\int_{\R^d}(1-\cos\langle \xi, z\rangle)\nu(\mbox{d}z)\neq 0$, and $M(\xi)=0$ otherwise.
\end{lemma}

Equipped with the above facts, we turn our attention to Theorem \ref{mainthmf}. 
The key ingredient in the proof of this statement is contained in the following. Let $U_p$ be the special function used in the proof of Theorem \ref{mainthm}.

\begin{lemma}\label{maincon}
Let $p\in (1,2)\cup (2,\infty)$. Then for any complex-valued function $f\in C_0^\infty(\R^d)$ we have the estimate
\begin{equation}\label{fly}
 \int_{\R^d} U_p(f(x),\mathcal{S}f(x))\mbox{d}x\leq 0, 
\end{equation}
where $U_p$ is the function defined in \eqref{U-p}. 
\end{lemma}
\begin{proof}
Take $g(x)=(U_{p})_{y}(f(x),\mathcal{S}f(x))$. Then we have $g\in L^{p/(p-1)}(\R^d)$, because $|(U_{p})_{y}(f(x),\mathcal{S}f(x))|\leq \eta_p(|f(x)|+|\mathcal{S}f(x)|)^{p-1}$ for some constant $\eta_p$ depending only on $p$. Therefore, by \eqref{defS} and Fubini's theorem, we have
$$ \E \int_{\R^d}\mathcal{S}f(x+X_{s,0})g(x+X_{s,0})\mbox{d}x=\E \int_{\R^d}\E \big[G_0^{x,s,f,\phi}g(x+X_{s,0})\big]\mbox{d}x,$$
or
$$ \E \int_{\R^d} (U_{p})_{y}(f(x+X_{s,0}),\mathcal{S}f(x+X_{s,0}))\big[G_0^{x,s,f,\phi}-\mathcal{S}f(x+X_{s,0})\big]\mbox{d}x=0.$$
Hence
\begin{align*}
&\int_{\R^d} U_p(f(x),\mathcal{S}f(x))\mbox{d}x\\
&=\E \int_{\R^d}U_p(f(x+X_{s,0}),\mathcal{S}f(x+X_{s,0}))\mbox{d}x\\
&=\E \int_{\R^d}\bigg\{U_p(f(x+X_{s,0}),\mathcal{S}f(x+X_{s,0}))\\
&\qquad \qquad \qquad +(U_{p})_{y}(f(x+X_{s,0}),\mathcal{S}f(x+X_{s,0}))\big[G_0^{x,s,f,\phi}-\mathcal{S}f(x+X_{s,0})\big]\bigg\}\mbox{d}x\\
&\leq \E \int_{\R^d}U_p(f(x+X_{s,0}),G_0^{x,s,f,\phi})\mbox{d}x\\
&=\int_{\R^d}\E U_p(F^{x,s,f}_0,G_0^{x,s,f,\phi})\mbox{d}x.
\end{align*}
We will be done if we prove that $\E U_p(F^{x,s,f}_0,G_0^{x,s,f,\phi})\leq 0$ for all $x$. This follows from Theorem \ref{Wang} and a limiting argument. Indeed, we know that there is a nondecreasing sequence $(\tau_n)_{n\geq 1}$ of stopping times converging to $0$ (and depending on $x$, $s$, $f$ and $\phi$) such that 
\begin{equation}\label{inmm}
\E U_p(F^{x,s,f}_{\tau_n\wedge 0},G_{\tau_n\wedge 0}^{x,s,f,\phi})\leq 0
\end{equation}
for all $n$. However, from the very definition of $U_p$, there is a constant $C_p>0$ such that 
$$ |U_p(x,y)|\leq C_p(|x|^p+|y|^p).$$
This implies 
$$U_p(F^{x,s,f}_{\tau_n\wedge 0},G_{\tau_n\wedge 0}^{x,s,f,\phi})\leq C_p (((F^{x,s,f})^*)^p+((G^{x,s,f,\phi})^*)^p)$$
(where $X^*$ denotes the maximal function of a martingale $X$). By Doob's inequality and Burkholder's estimate \eqref{burkin}, we see that
$$ \E((F^{x,s,f})^*)^p+((G^{x,s,f})^*)^p)\leq \left(\frac{p}{p-1}\right)^p(1+(p^*-1)^p)\E |F^{x,s,f}_0|^p<\infty,$$
since $||F^{x,s,f}||_\infty\leq ||f||_\infty$. It remains to let $n\to \infty$ in \eqref{inmm} and use Lebesgue's dominated convergence theorem to get the claim.
\end{proof}

We are ready to establish the stability result for Fourier multipliers.

\begin{proof}[Proof of \eqref{mainf<2} and \eqref{mainf>2}]
It suffices to prove the inequality for bounded $f$. We will only give the details for $1<p<2$, in the remaining case the reasoning is analogous. By \eqref{fly} and the majorization \eqref{maj<2} we get
\begin{align*}
 \left(1-\left(1-\frac{1}{p}\right)^{p-1}\right)\int_{\R^d}\frac{((p-1)|\mathcal{S}f(x)|-|f(x)|)^2}{(|f(x)|+|\mathcal{S}f(x)|)^{2-p}}
+(p-1)^p||&\mathcal{S}f(x)||_{L^p(\R^d)}^p\\
&\leq ||f||_{L^p(\R^d)}^p.
\end{align*}
Recall that $\mathcal{S}=\mathcal{S}^{s,\phi,\nu}$ is a Fourier multiplier with the symbol given by \eqref{defMs}. If we let $s\to -\infty$, then the symbol $M_{s,\phi,\nu}$ converges pointwise to 
\begin{equation}\label{defM}
M_{\phi,\nu}(\xi)=\frac{\int_{\R^d}(1-\cos\langle \xi, z\rangle)\phi(z)\nu(\mbox{d}z)}{\int_{\R^d}(1-\cos\langle \xi, z\rangle)\nu(\mbox{d}z)}.
\end{equation}
By Plancherel's theorem, $\mathcal{S}^{s,\phi,\nu}f=T_{M_{s,\phi,\nu}}f \to T_{M_{\phi,\nu}}f$ in $L^2$ and hence there is a sequence $(s_n)_{n=1}^\infty$ converging to $-\infty$ such that  $\lim_{n\to\infty}S^{s_n,\phi,\nu}f \to T_{M_{\phi,\nu}}f$ almost everywhere. Thus Fatou's lemma combined with the above estimate yields 
\begin{equation}\label{innere}
\begin{split}
 \left(1-\left(1-\frac{1}{p}\right)^{p-1}\right)&\int_{\R^d}\frac{((p-1)|T_{M_{\phi,\nu}}f(x)|-|f(x)|)^2}{(|f(x)|+|T_{M_{\phi,\nu}}f(x)|)^{2-p}}
\\&+(p-1)^p||T_{M_{\phi,\nu}}f(x)||_{L^p(\R^d)}^p\leq ||f||_{L^p(\R^d)}^p.
\end{split}
\end{equation}
Now, for a given $\kappa>0$, define a L\'evy measure $\nu_\kappa$ in polar coordinates $(r,\theta)\in (0,\infty)\times \mathbb{S}$ by
$$ \nu_\kappa(\mbox{d}r\mbox{d}\theta)=\kappa^{-2}\delta_\kappa(\mbox{d}r)\mu(d\theta),$$
where $\delta_\kappa$ stands for the Dirac measure on $\{\kappa\}$ and $\mu$ is the measure involved in the definition of the symbol $m$. Next, consider a multiplier $m_\kappa$ as in \eqref{defM}, in which the L\'evy measure is $\nu_\kappa$ and the jump modulator is  given by $1_{\{|x|=\kappa\}}\psi(x/|x|)$ ($\psi$ is the function involved in the definition of $m$). If we let $\kappa\to 0$, we see that
\begin{equation*}
\begin{split}
 \int_{\R^d}[1-\cos\langle \xi,x\rangle]\psi(x/|x|)\nu_\kappa(\mbox{d}x)&=\int_{\mathbb{S}}\psi(\theta)\frac{1-\cos\langle \xi,\kappa\theta\rangle}{\kappa^2}\mu(d\theta)\\
 &\to \frac{1}{2}\int_{\mathbb{S}}\langle \xi,\theta\rangle^2\psi(\theta)\mu(\mbox{d}\theta)
 \end{split}
 \end{equation*}
and similarly
$$ \int_{\R^d}[1-\cos\langle \xi,x\rangle]\psi(x/|x|)\nu_\kappa(\mbox{d}x)\to \frac{1}{2}\int_{\mathbb{S}}\langle \xi,\theta\rangle^2\mu(\mbox{d}\theta),$$
so that $m_\kappa\to m$ pointwise. 
 This observation, combined with \eqref{innere}, yields the estimate
\begin{align*}
 \left(1-\left(1-\frac{1}{p}\right)^{p-1}\right)&\int_{\R^d}\frac{((p-1)|T_mf(x)|-|f(x)|)^2}{(|f(x)|+|T_mf(x)|)^{2-p}}
\\&+(p-1)^p||T_mf(x)||_{L^p(\R^d)}^p\leq ||f||_{L^p(\R^d)}^p,
\end{align*}
by the similar argument as above, using of Plancherel's theorem and the passage to the subsequence which converges almost everywhere. In other words, we have
\begin{align*}
 &\left(1-\left(1-\frac{1}{p}\right)^{p-1}\right)\int_{\R^d}\frac{((p-1)|T_mf(x)|-|f(x)|)^2}{(|f(x)|+|T_mf(x)|)^{2-p}}\\
&\qquad \qquad\qquad \qquad \leq ||f||_{L^p(\R^d)}^p-(p-1)^p||T_mf(x)||_{L^p(\R^d)}^p\\
&\qquad \qquad\qquad \qquad \leq (1-(1-(p-1)\e)^p)||f||_{L^p(\R^d)}^p\\
&\qquad \qquad\qquad \qquad \leq p(p-1)^{p-1}\e ||f||_{L^p(\R^d)}^p.
\end{align*}
Therefore, H\"older's inequality and the estimate \eqref{BanBog} imply
\begin{align*}
&\big|\big|(p-1)|T_mf|-|f|\big|\big|_{L^p(\R^d)}\\
&\leq \left(\int_{\R^d} \frac{((p-1)|T_mf|-|f|)^2}{(|f|+|T_mf|)^{2-p}}\right)^{1/2}|||f|+|T_mf|||_{L^p(\R^d)}^{1-p/2}\\
&\leq \left(\frac{p(p-1)\e}{1-p\left(1-\frac{1}{p}\right)^{p-1}}\right)^{1/2} ||f||_{L^p(\R^d)}^{p/2}\cdot \left(\frac{p}{p-1}||f||_{L^p(\R^d)}\right)^{1-p/2}.
\end{align*}
This is the claim.
\end{proof}

\begin{rem}
The above argumentation can be easily carried over to the vector-valued case. Let us state this more precisely. Suppose that $f=(f_1,f_2,\ldots)\in L^p(\R^d;\mathcal{H})$, i.e., $f$ is a $p$-integrable function on $\R^d$ taking values in $\mathbb{H}$. Let $m=(m_1,m_2,\ldots)$, where for any $j$,  $m_j$ is a symbol from the class \eqref{defm}, with the corresponding parameters $\phi_j$ and $\mu_j$. We define the Fourier multiplier $T_m$, associated with $m$, by the coordinate-wise action: $T_mf=(T_{m_1}f_1,T_{m_2}f_2,\ldots)$. Then the inequalities of Theorem \ref{mainthmf} hold true under this more general setting. Indeed, one fixes $s<0$ and introduces the $\mathbb{H}$-valued martingales $F$ and $G$, as well as the ``intermediate'' operator $\mathcal{S}=(\mathcal{S}^{s,\phi_1,\nu_1},\mathcal{S}^{s,\phi_2,\nu_2},\ldots)$, where each $\nu_j$ is a L\'evy measure on $\R^d$. If one writes \eqref{defS} for each $j$ (and some functions $g_j$) and sums the obtained identities, one gets
$$ \int_{\R^d}\langle \mathcal{S}f(x),g(x)\rangle\mbox{d}x=\int_{\R^d}\E \langle G_0^{x,s,f,\phi},g(x+X_{s,0})\rangle\mbox{d}x.$$
Having done this, one easily shows the vector-valued version of the inequality \eqref{fly} just by replacing products appearing under integrals by inner products of the corresponding vectors. The remainder of the proof is a word-by-word repetition of the arguments used in the scalar case.
\end{rem}

\subsection{Sharpness of 
Theorem \ref{mainthmf}}

\def\diag{\operatorname*{diag}}
We will now construct appropriate functions showing that the order of constants involved in \eqref{mainf<2} and \eqref{mainf>2} is quite tight. 
 Our approach depends heavily on the paper \cite{BSV} by Boros, Sz\'ekelyhidi and Volberg, in which the interplay between martingale transforms and the class of the so-called laminates, important probability measures on matrix spaces (see below) was investigated for the first time. In order to make this section as self-contained as possible, we recall all the basic information on the subject. 

Let  $\R^{m\times n}$ denote  the space of all real matrices of dimension  $m\times n$ and  $\R^{n\times n}_{sym}$ denote  the subclass of $\R^{n\times n}$ consisting  of all real symmetric $n\times n$ matrices.

\begin{dfn}
A function $f:\R^{m\times n} \to \R$ is said to be {\it rank-one convex}, if  for all $A,B \in \R^{m\times n}$ with $\textrm{rank }B= 1$, the function 
$t\mapsto f(A+tB)$ is convex.
\end{dfn}

Let $\mathcal{P}=\mathcal{P}(\R^{m\times n})$ denote the class of all compactly supported probability measures on the space $\R^{m \times n}$.
For $\nu \in \mathcal{P}$, the \emph{center of mass}, or \textit{barycenter} of $\nu$, is given by
$$\overline{\nu} = \int_{\R^{m\times n}}{X d\nu(X)}$$ 

\begin{dfn}
We say that a measure $\nu \in \mathcal{P}$ is a \textit{laminate} (and write $\nu\in\mathcal{L}$), if 
\begin{equation*}
f(\overline{\nu}) \leq \int_{\R^{m\times n}}f \mbox{d}\nu
\end{equation*} 
for all rank-one convex functions $f$. The set of laminates with barycenter $0$ (the zero matrix) is denoted by $\mathcal{L}_0$. 
\end{dfn}
Laminates arise naturally in several applications of convex integration, where they can be used to produce interesting counterexamples;  see  \cite{AFS}, \cite{CFM}, \cite{KMS}, \cite{MS99} and \cite{SzCI}. For our results in this paper we will be interested in the case of $2\times 2$ symmetric matrices. An important observation to make is  that laminates can be regarded as probability measures that record the distribution of the gradients of smooth maps as described by Corollary \ref{coro} below. We briefly explain this and refer the reader to \cite{Kirchheim}, \cite{MS99} and \cite{SzCI} for the full discussion.  

\begin{dfn}
Let $\mathcal{PL}$ denote  the smallest class of probability measures on $\R^{2\times 2}_{sym}$ which 

\begin{itemize}

\item[(i)] contains all measures of the form $\lambda \delta_A+(1-\lambda)\delta_B$ with $\lambda\in [0,1]$ and satisfying $\textrm{rank}(A-B)=1$;

\item[(ii)] is closed under splitting in the following sense: if $\lambda\delta_A+(1-\lambda)\nu$ belongs to $\mathcal{PL}$ for some $\nu\in\mathcal{P}(\R^{2\times 2})$ and $\mu$ also belongs to $\mathcal{PL}$ with $\overline{\mu}=A$, then also $\lambda\mu+(1-\lambda)\nu$ belongs to $\mathcal{PL}$.
\end{itemize}

The class $\mathcal{PL}$ is called the class of \emph{prelaminates}. 
\end{dfn} 

It is clear from the very definition that the class $\mathcal{PL}$ contains only atomic measures. Also, by a successive application of Jensen's inequality, we have the inclusion $\mathcal{PL}\subset\mathcal{L}$. Recall the following two well-known results in the theory of laminates; see \cite{AFS}, \cite{Kirchheim}, \cite{MS99}, \cite{SzCI}.

\begin{lemma}
Let $\nu=\sum_{i=1}^N\lambda_i\delta_{A_i}\in\mathcal{PL}$ with $\overline{\nu}=0$. Moreover, let
$0<r<\tfrac{1}{2}\min|A_i-A_j|$ and $\delta>0$. For any bounded domain $\calB\subset\R^2$ there exists $u\in W^{2,\infty}_0(\calB)$ such that $\|u\|_{C^1}<\delta$ and for all $i=1\dots N$
$$
\bigl|\{x\in\calB:\,|D^2u(x)-A_i|<r\}\bigr|=\lambda_i|\calB|.
$$
\end{lemma}

\begin{lemma}
Let $K\subset\R^{2\times 2}_{sym}$ be a compact convex set and $\nu\in\mathcal{L}$ with $\operatorname*{supp}\nu\subset K$. Then there exists a sequence $\nu_j$ of prelaminates with $\overline{\nu}_j=\overline{\nu}$ and $\nu_j\overset{*}{\rightharpoonup}\nu$, where $\overset{*}{\rightharpoonup}$ denotes weak convergence of measures. 
\end{lemma}

Combining these two lemmas and using a simple mollification, we obtain the following statement, proved by Boros, Sh\'ekelyhidi Jr. and Volberg \cite{BSV}. It exhibits the connection  between laminates supported on symmetric matrices and second derivatives of functions.  This fact will play a crucial role in our argumentation below. As in the introduction, the symbol $\mathbb{D}$ stands for the unit disc in the complex plane  $\mathbb{C}$.

\begin{corollary}\label{coro}
Let $\nu\in\mathcal{L}_0$. Then there exists a sequence $u_j\in C_0^{\infty}(\mathbb{D})$ with uniformly bounded second derivatives, such that
\begin{equation}\label{gby}
\frac{1}{|\mathbb{D}|}\int_{\mathbb{D}} \phi(D^2u_j(x))\,\mbox{d}x\,\to\,\int_{\R^{2\times 2}_{sym}}\phi\,\mbox{d}\nu
\end{equation}
for all continuous $\phi:\R^{2\times 2}_{sym}\to\R$. 
\end{corollary}

We are ready to provide lower bounds for the constants and exponents involved in the estimates \eqref{mainf<2} and \eqref{mainf>2}, and prove that there is no stability result in the case $p=2$. We will focus on the case $1<p<2$.  For the remaining cases the reasoning is essentially the same and we leave it to the reader. 

Fix $1<p<2$, a small $\e>0$, a large $K>0$ and set $\eta=(p-1)\sqrt{\e/(2-p)}/2$.  We may assume that \eqref{ineq} holds by decreasing $\e$ if necessary. Let $(F,G)$ be the (finite) martingale pair studied in  \S\ref{Sec<2}. 

\begin{lemma}
The distribution of $\diag(G_\infty-F_\infty,F_\infty+G_\infty)\in \R^{2\times 2}_{sym}$ is a prelaminate with barycenter $0$.
\end{lemma}
\begin{proof}
By the construction, the martingale $F$ is a transform of $G$ by the deterministic sequence $\{(-1)^n\}_{n\geq 0}$. Consequently, the pair $(G-F,F+G)$ has the following zigzag property: for any $n$, it moves either vertically or horizontally. More precisely, depending on the parity of $n$, we have $G_{n+1}-F_{n+1}=G_n-F_n$ with probability $1$ or $F_{n+1}+G_{n+1}=F_n+G_n$ with probability $1$. This implies the desired prelaminate property: when comparing the distributions of $(G_{n}-F_{n},F_n+G_n)$ and $(G_{n+1}-F_{n+1},F_{n+1}+G_{n+1})$ we see exactly the splitting as in the second condition defining the class of prelaminates.
\end{proof}

Next, consider the continuous functions $\phi_1,\,\phi_2:\R^{2\times 2}_{sym}\to \R$ given by
$$ \phi_1(A)=|A_{11}-A_{22}|^p-\left(\frac{1}{p-1}-\e\right)^p|A_{11}+A_{22}|^p$$
and
\begin{eqnarray*}
\phi_2(A)&=&\big||A_{11}-A_{22}|-(p-1)^{-1}|A_{11}+A_{22}|\big|^p\\
&-&\left(\frac{p}{4(p-1)}\sqrt{\frac{\e}{2-p}}\right)^p|A_{11}+A_{22}|^p.
\end{eqnarray*}
By Corollary \ref{coro}, there is a sequence  $u_j\in C_0^{\infty}(\mathbb{D})$  such that
$$ \frac{1}{|\mathbb{D}|}\int_{\mathbb{D}} \phi_i(D^2u_j(x))\,\mbox{d}x\,\to\,\int_{\R^{2\times 2}_{sym}}\phi_i\,\mbox{d}\nu,\qquad i=1,\,2.
$$
If we set $f_j=\Delta u_j$, this equivalent to the statement that 
\begin{align*}
 &\frac{1}{|\mathbb{D}|}\int_{\mathbb{D}} \left(|\Re B f_j|^p-\left(\frac{1}{p-1}-\e\right)^p|f_j|^p\right)\mbox{d}x\\
&\qquad \qquad \qquad \to\,2^p\left[||F_\infty||_p^p-\left(\frac{1}{p-1}-\e\right)^p||G_\infty||^p_p \right]>0
\end{align*}
and
\begin{align*}
 &\frac{1}{|\mathbb{D}|}\int_{\mathbb{D}} \left(\big||\Re B f_j|-(p-1)^{-1}|f_j|\big|^p-\left(\frac{p}{4(p-1)}\sqrt{\frac{\e}{2-p}}\right)^p|f_j|^p\right)\mbox{d}x\\
&\qquad \qquad \to\,2^p\left[\big|\big||F_\infty|-(p-1)^{-1}|G_\infty|\big|\big|_p^p-\left(\frac{p}{4(p-1)}\sqrt{\frac{\e}{2-p}}\right)^p||G_\infty||_p^p\right]>0.
\end{align*}
Therefore, for sufficiently large $j$ we have
$$ ||\Re Bf_j||_{L^p(\mathbb{C})}\geq \left(\int_{\mathbb{D}}|\Re Bf_j|^p\mbox{d}x\right)^{1/p}\geq \left(\frac{1}{p-1}-\e\right)||f_j||_{L^p(\mathbb{C})}$$
and, simultaneously,
$$ \big|\big||\Re B f_j|-(p-1)^{-1}|f_j|\big|\big|_{L^p(\mathbb{C})}\geq \left(\frac{p}{4(p-1)}\sqrt{\frac{\e}{2-p}}\right)||f_j||_{L^p(\mathbb{C})}.$$
This is precisely the desired bound.

\subsection{First order Riesz transforms, inequalities \eqref{mainR<2} and \eqref{mainR>2} in Theorem \ref{mainthmR}}
The reasoning is similar to that above, so we will be brief; we will mostly focus on the case $p<2$, for other values of $p$ we proceed analogously. 
Our argumentation rests on the well-known representation of Riesz transforms in terms of the so-called background radiation process, introduced by Gundy and Varopoulos in \cite{GV}. Let us briefly describe this connection. Throughout this section, $d$ is a fixed positive integer. Suppose that $X$ is a Brownian motion in $\R^d$ and let $Y$ be an independent Brownian motion in $\R$ (both processes start from the appropriate origins). For any $y>0$, introduce the stopping time $\tau(y)=\inf\{t\geq 0: Y_t\in \{-y\}\}$. If $f$ belongs to $\mathcal{S}(\R^d)$, the class of rapidly decreasing functions on $\R^d$, let $Z_f:\R^d\times [0,\infty)\to \R$ 
stand for the Poisson extension of $f$ to the upper half-space. That is,
$$Z_f(x,y):=\E f\left(x+X_{\tau(y)}\right).$$
For any $(d+1)\times (d+1)$ matrix $A$ we define the martingale transform $A\!*\!f$ by
$$ A\!*\! f(x,y)=\int_{0+}^{\tau(y)} A\nabla Z_f(x+X_s,y+Y_s)\cdot \mbox{d}(X_s,Y_s).$$
Note that $A*f(x,y)$ is a random variable for each $x,\,y$. 
Now, for any $f\in C_0^\infty$, any $y>0$ and any matrix $A$ as above, define $\mathcal{T}_A^yf:\R^d\to \R$ through the bilinear form
\begin{equation}\label{defT}
\int_{\R^d} \mathcal{T}_A^yf(x)g(x)\,\mbox{d}x=\int_{\R^d} \E \big[A\!*\!f(x,y)g(x+X_{\tau(y)})\big]\mbox{d}x,
\end{equation}
where $g$ runs over $C_0^\infty(\R^d)$. The interplay between the operators $\mathcal{T}_A^y$ and Riesz transforms is explained in the following theorem, consult \cite{GV} or Gundy and Silverstein \cite{GS}.

\begin{theorem}\label{RieszGV}
Let $A^j=[a^j_{\ell m}]$, $j=1,\,2,\,\ldots,\,d$ be the $(d+1)\times (d+1)$ matrices given by
$$ a^j_{\ell m}=\left\{\begin{array}{ll}
1 & \mbox{if }\ell=d+1,\,m=j,\\
-1 & \mbox{if }\ell=j,\,m=d+1,\\
0 & \mbox{otherwise}.
\end{array}\right.$$
Then $\mathcal{T}_{A^j}^yf \to R_jf$ almost everywhere as $y\to \infty$.
\end{theorem}

Here is the analogue of Lemma \ref{maincon}, exploiting the functions $V_p$ introduced in the proof of Theorem \ref{mainthmo}. The argument goes along the same lines, so we will not repeat it here.

\begin{lemma}
Let $p\in (1,2)\cup (2,\infty)$ and let $A$ be the matrix from Theorem \ref{RieszGV} corresponding to Riesz transform $R_j$. Then for any real-valued function $f\in C_0^\infty(\R^d)$ we have the estimate
\begin{equation}\label{fly2}
 \int_{\R^d} V_p(f(x),\mathcal{T}_Af(x))\mbox{d}x\leq 0.
\end{equation}
\end{lemma}

\begin{proof}[Proof of \eqref{mainR<2}] 
By \eqref{fly2} and the majorization \eqref{majo<2}, we get
\begin{align*}
 \kappa_p \int_{\R^d}\left(|\mathcal{T}_Af(x)|-\tan\frac{\pi}{2p}|f(x)|\right)^2(|f(x)|+|\mathcal{T}_Af(x)|)^{2-p}\mbox{d}x&+||\mathcal{T}_Af||^p_{L^p(\R^d)}\\
&\leq \tan^p\frac{\pi}{2p}||f||_{L^p(\R^d)}^p.
\end{align*}
Now we exploit Theorem \ref{RieszGV}: if we let the parameter $y$ go to $\infty$, then Fatou's lemma implies
\begin{align*}
 \kappa_p \int_{\R^d}\left(|R_jf(x)|-\tan\frac{\pi}{2p}|f(x)|\right)^2(|f(x)|+|R_jf(x)|)^{2-p}\mbox{d}x&+||R_jf||^p_{L^p(\R^d)}\\
&\leq \tan^p\frac{\pi}{2p}||f||_{L^p(\R^d)}^p.
\end{align*}
Since $||R_jf||_{L^p(\R^d)}\geq (\tan^p\frac{\pi}{2p}-\e)||f||_{L^p(\R^d)}$, we obtain
\begin{align*}
 \kappa_p \int_{\R^d}\left(|R_jf(x)|-\tan\frac{\pi}{2p}|f(x)|\right)^2&(|f(x)|+|R_jf(x)|)^{2-p}\mbox{d}x\\
&\leq p\tan^{p-1}\frac{\pi}{2p}\e||f||_{L^p(\R^d)},
 \end{align*}
which combined with H\"older inequality and \eqref{inn} yields
\begin{align*}
\left|\left||R_jf|-\tan\frac{\pi}{2p}|f|\right|\right|_{L^p(\R^d)}&\leq \left(\E \frac{(|R_jf|-\tan\frac{\pi}{2p}|f|)^2}{(|f|+|R_jf|)^{2-p}}\right)^{1/2}|||f|+|R_jf|||_{L^p(\R^d)}^{1-p/2}\\
&\leq \left(\frac{p\tan^{p-1}\frac{\pi}{2p}\e}{\kappa_p}\right)^{1/2}  \left(1+\tan\frac{\pi}{2p}\right)^{1-p/2}||f||_{L^p(\R^d)}\\
&= \left(-\frac{8p\tan\frac{\pi}{2p}\e}{(p-1)\cos\frac{\pi}{p}}\right)^{1/2}  \left(1+\tan\frac{\pi}{2p}\right)^{1-p/2}||f||_{L^p(\R^d)}.
\end{align*}
Now it suffices to apply the same bounds as at the end of the proof of \eqref{maino<2} to get the claim.
\end{proof}

 \subsection{Shapness of Theorem \ref{mainthmR}}
\begin{proof}[Sharpness of \eqref{mainR<2}, $d=1$]
Fix $1<p<2$ and $\e>0$. We have constructed above a pair $(X,Y)$ of orthogonal martingales such that $||Y||_p> (\tan\frac{\pi}{2p}-\e)||X||_p$ and $|||Y_\infty|-\tan\frac{\pi}{2p}|X_\infty||_p\geq a_p\left(\frac{\e}{2-p}\right)^{1/2}||X||_p$ for some constant $a_p$ bounded in a neighborhood of $2$. Actually, the pair $(X,Y)$ was the planar Brownian motion started at the origin, killed upon leaving the boundary of a certain angle $D$. Let $F:\mathbb{D}\to D$ be the conformal map which sends the unit disc of the complex plane onto that angle, such that $F(0)=0$. Then the distribution, with respect to the Haar measure, of $F=\Re{F}+i\Im F=\Re{F}+i\mathcal{H}^\mathbb{T}\Re{F}$ on the unit circle $\mathbb{T}$ coincides with the distribution of the pair $(X,Y)$; therefore, we have
$$ ||\mathcal{H}^\mathbb{T}\Re{F}||_{L^p(\mathbb{T})}> \left(\tan\frac{\pi}{2p}-\e\right)||f||_{L^p(\mathbb{T})}$$
and, at the same time,
$$\left|\left||\mathcal{H}^\mathbb{T}\Re{F}|-\tan\frac{\pi}{2p}|\Re{F}|\right|\right|_{L^p(\mathbb{T})}\geq a_p\left(\frac{\e}{2-p}\right)^{1/2}||\Re F||_{L^p(\mathbb{T})}.$$
To pass from the periodic Hilbert transform $\mathcal{H}^\mathbb{T}$ to its non-periodic counterpart, we exploit well-known argument going back to Davis' work \cite{Da}. Let $\textrm{H}$ denote the upper half-plane and let $G:D\cap \textrm{H}\to \textrm{H}$ be defined by $G(z)=-(1-z)^2/(4z)$. Then $G$ is conformal and hence so is its inverse $L$. We extend $L$ to the continuous function on $\overline{\textrm{H}}=\{z\in \mathbb{C}:\mbox{Im}z\geq 0\}$. Then $L$ maps $[0,1]$ onto $\{e^{i\theta}:0\leq \theta\leq \pi\}$; specifically, for $x\in [0,1]$ we have
\begin{equation}\label{soul}
L(x)=e^{i\theta}, \mbox{ where $\theta \in [0,\pi]$ is uniquely determined by $x=\sin^2(\theta/2)$.}
\end{equation}
Moreover, $L$ maps $\R\setminus [0,1]$ onto $(-1,1)$; precisely, we have
\begin{equation}\label{soul2}
L(x)=\left\{\begin{array}{ll}
1-2x-2\sqrt{x^2-x} & \mbox{if }x<0,\\
1-2x+2\sqrt{x^2-x} & \mbox{if }x>1.
\end{array}\right.
\end{equation}
Therefore, if we take $\varphi_n=\Re\big(F(L^{2n})\big)$, $x\in \R$, then $\mathcal{H}^\R\varphi_n=\Im\big(F(L^{2n})\big)$, since $F(L^{2n})$ is analytic on $\textrm{H}$ and it vanishes at $\infty$ (the latter follows from the requirement $F(0)=0$). Using \eqref{soul}, we derive that
\begin{align*}
 \int_\R |\mathcal{H}^\R \varphi_n(x)|^p\mbox{d}x&\geq \int_0^1 |\Im\big(F(L^{2n})\big)|^p\mbox{d}x\\
 &=\frac{1}{2}\int_0^\pi|\Im\big(F(e^{2in\theta})\big)|^p\sin\theta\mbox{d}\theta\\
&=\frac{1}{2}\int_0^{2n\pi}|\Im\big(F(e^{i\theta})\big)|^p\sin\left(\frac{\theta}{2n}\right)\frac{\mbox{d}\theta}{2n}\\
&=\frac{1}{2}\int_0^{2\pi}|\Im\big(F(e^{i\theta})\big)|^p\sum_{k=0}^{n-1}\sin\left(\frac{k\pi}{n}+\frac{\theta}{2n}\right)\frac{\mbox{d}\theta}{2n}\\
&=\frac{1}{2}\int_0^{2\pi}|\Im\big(F(e^{i\theta})\big)|^p\frac{\cos\left(\frac{\theta-\pi}{n}\right)}{2n\sin\left(\frac{\pi}{2n}\right)}\mbox{d}\theta\\
&\xrightarrow{n\to\infty} \frac{1}{2\pi}\int_0^{2\pi}|\Im\big(F(e^{i\theta})\big)|^p\mbox{d}\theta=||Y||_p^p
\end{align*}
and similarly
$$ \int_0^1 |\varphi_n(x)|^p\mbox{d}x\xrightarrow{n\to\infty}||X||_p^p.$$
Furthermore, exploiting \eqref{soul2} and the condition $F(0)=0$, we easily get
\begin{align*}
\int_{\R\setminus [0,1]} |\varphi_n(x)|^p\mbox{d}x\xrightarrow{n\to\infty} 0
\end{align*}
and
\begin{align*}
 &\left(\int_\R \left||\mathcal{H}^\R \varphi_n(x)|-\tan\frac{\pi}{2p}|\varphi_n(x)|\right|^p\mbox{d}x\right)^{1/p}\\
&\qquad \qquad \qquad \geq 
 \left(\int_0^1 \left||\mathcal{H}^\R \varphi_n(x)|-\tan\frac{\pi}{2p}|\varphi_n(x)|\right|^p\mbox{d}x\right)^{1/p}\\
 &\qquad \qquad \qquad \xrightarrow{n\to \infty} \left|\left||Y_\infty|-\tan\frac{\pi}{2p}|X_\infty|\right|\right|_p.
 \end{align*}
This proves the desired optimality of the constants. For $p>2$ the reasoning is essentially the same and we leave it to the interested reader.
\end{proof}

\begin{proof}[Sharpness of \eqref{mainR<2} and \eqref{mainR>2}, the case $d>1$] Clearly, it is enough to handle the Riesz transform $R_1$ only. Fix $p\in (1,2)\cup(2,\infty)$ and suppose that there is a nondecreasing function $\varphi_p:[0,\infty)\to (0,\infty)$ such that for all $f\in L^p(\R^d)$ we have
\begin{equation}\label{LL}
 \left|\left| |R_1f|-\cot\frac{\pi}{2p^*}|f|\right|\right|_p\leq \varphi_p\left(\cot\frac{\pi}{2p^*}-\frac{||R_1 f||_{L^p(\R^d)}}{||f||_{L^p(\R^d)}}\right)||f||_{L^p(\R^d)}.
\end{equation}
Our plan is to show that this inequality implies the validity of the corresponding estimate for the Hilbert transform on the real line (with the same function $\varphi_p$). This will clearly yield the announced optimality of the constants appearing in \eqref{mainR<2} and \eqref{mainR>2}, by the case $d=1$ considered above. 
For $t>0$, define the dilation operator $\delta_t$ as follows: for any function $g:\R\times \R^{d-1}\to \R$, we let $\delta_tg(\xi,\zeta)=g(\xi,t\zeta)$. Using \eqref{LL}, we see that the operator $T_t:=\delta_t^{-1}\circ R_1\circ \delta_t$ satisfies
\begin{equation}\label{LLL}
\begin{split}
&\left|\left| |T_tf|-\cot\frac{\pi}{2p^*}|f|\right|\right|_{L^p(\R^d)}\\
&=t^{(d-1)/p}\left|\left| |R_1\circ \delta_tf|-\cot\frac{\pi}{2p^*}|\delta_tf|\right|\right|_{L^p(\R^d)}\\
&\leq t^{(d-1)/p}\varphi_p\left(\cot\frac{\pi}{2p^*}-\frac{||R_1\circ \delta_t f||_{L^p(\R^d)}}{||\delta_tf||_{L^p(\R^d)}}\right)||\delta_tf||_{L^p(\R^d)}\\
&= \varphi_p\left(\cot\frac{\pi}{2p^*}-\frac{||T_t f||_{L^p(\R^d)}}{||f||_{L^p(\R^d)}}\right)||f||_{L^p(\R^d)}.
\end{split}
\end{equation}
It is easy to check that the Fourier transform $\mathcal{F}$ satisfies the identity $\mathcal{F}=t^{d-1}\delta_t\circ \mathcal{F}\circ \delta_t$ and therefore the operator 
$T_t$ has the property
$$ \widehat{T_tf}(\xi,\zeta)=-i\frac{\xi}{(\xi^2+t^2|\zeta|^2)^{1/2}}\widehat{f}(\xi,\zeta),\qquad (\xi,\zeta)\in \R\times \R^{d-1},$$
for any square integrable $f$ on $\R^d$. By Lebesgue's dominated convergence theorem, we have
$$ \lim_{t\to 0}\widehat{T_tf}(\xi,\zeta)=\widehat{T_0f}(\xi,\zeta)$$
in $L^2(\R^d)$, where $\widehat{T_0f}(\xi,\zeta)=-i\,\mbox{sgn}\,(\xi)\widehat{f}.$ 
Combining this with Plancherel's theorem, we obtain that for any $f\in L^2(\R^d)$ there is a sequence $(t_n)_{n\geq 1}$ decreasing to $0$ such that $T_{t_n}f$ converges to $T_0f$ almost everywhere. Using Fatou's lemma, \eqref{LLL} and the monotonicity of $\varphi_p$, we obtain
\begin{equation}\label{L4}
 \left|\left| |T_0f|-\cot\frac{\pi}{2p^*}|f|\right|\right|_{L^p(\R^d)}\leq \varphi_p\left(\cot\frac{\pi}{2p^*}-\frac{||T_0 f||_{L^p(\R^d)}}{||f||_{L^p(\R^d)}}\right)||f||_{L^p(\R^d)}
\end{equation}
Note that $T_t$ are bounded on $L^p(\R^d)$ for $1<p<\infty$ (in fact, $||T_t||_{{L^p(\R^d)}\to{L^p(\R^d)}}=||R_1||_{{L^p(\R^d)}\to {L^p(\R^d)}}$), hence so is $T_0$ and thus the above estimate holds true for all $f\in L^p(\R^d)$. 
Define $f:\R\times \R^{d-1}\to \R$ by $f(\xi,\zeta)=h(\xi)1_{[0,1]^{d-1}}(\zeta),$ where $h$ is an arbitrary function belonging to $L^p(\R)$. Then $f\in L^p(\R^d)$ and $T_0f(\xi,\zeta)=\mathcal{H}^\R h(\xi)1_{[0,1]^{d-1}}(\zeta)$, which is due to the identity
$$ \widehat{T_0f}(\xi,\zeta)=-i\mbox{sgn}\,(\xi)\,\widehat{h}(\xi)\widehat{1_{[0,1]^{d-1}}}(\zeta).$$
Plug this into \eqref{L4} to obtain
$$  \left|\left| |\mathcal{H}^\R h|-\cot\frac{\pi}{2p^*}|h|\right|\right|_{L^p(\R^d)}\leq \varphi_p\left(\cot\frac{\pi}{2p^*}-\frac{||T_0 h||_{L^p(\R^d)}}{||h||_{L^p(\R^d)}}\right)||h||_{L^p(\R^d)}.$$ 
This yields the desired sharpness.
\end{proof}

\end{document}